\documentclass[11pt,letterpaper]{amsart}
\pdfoutput=1
\usepackage{eucal}
\usepackage{graphicx}
\usepackage{amsopn,amssymb}
\usepackage{hyperref}
\hypersetup{pdfauthor={David Dumas},pdftitle={Holonomy limits of complex projective structures},colorlinks=true,linkcolor=black,citecolor=black}

\usepackage{caption}
\newcommand{\into}{\hookrightarrow}
\newcommand{\tensor}{\otimes}
\newcommand{\id}{\mathrm{Id}}
\newcommand{\Z}{\mathbb{Z}}
\newcommand{\R}{\mathbb{R}}
\newcommand{\C}{\mathbb{C}}
\newcommand{\N}{\mathbb{N}}
\newcommand{\B}{\mathrm{B}}
\newcommand{\grad}{\nabla}
\renewcommand{\flat}{\beta}

\newcommand{\sR}{\mathcal{R}}
\DeclareMathOperator{\im}{Im}
\DeclareMathOperator{\re}{Re}
\DeclareMathOperator{\stab}{Stab}
\renewcommand{\Re}{\re}
\renewcommand{\Im}{\im}
\newcommand{\sslash}{/\!\!/}
\newcommand{\define}{:=}

\newcommand{\F}{\mathcal{F}}
\newcommand{\CP}{\mathbb{CP}}
\renewcommand{\H}{\mathbb{H}}
\newcommand{\ML}{\mathcal{ML}}
\renewcommand{\P}{\mathbb{P}}
\newcommand{\PQ}{\mathbb{P}^+Q}
\renewcommand{\tilde}[1]{\widetilde{#1}}
\newcommand{\X}{\mathcal{X}}

\newcommand{\hyp}{\sigma_{\mathrm{hyp}}}
\newcommand{\sph}{\sigma_{\CP^1}}
\newcommand{\I}{\mathrm{I}}
\newcommand{\II}{\mathrm{I\!I}}
\newcommand{\CAT}{\mathrm{CAT}}
\newcommand{\eps}{\varepsilon}
\renewcommand{\sl}{\mathfrak{sl}}
\renewcommand{\leq}{\leqslant}
\renewcommand{\geq}{\geqslant}

\newcommand{\noproof}{\hfill\qedsymbol}

\DeclareMathOperator{\olim}{\lim_{\omega}}

\DeclareMathOperator{\Lie}{\mathrm{Lie}}
\DeclareMathOperator{\Ad}{\mathrm{Ad}}
\DeclareMathOperator{\PSL}{\mathrm{PSL}}
\DeclareMathOperator{\SL}{\mathrm{SL}}
\DeclareMathOperator{\SU}{\mathrm{SU}}
\DeclareMathOperator{\hol}{hol}
\DeclareMathOperator{\isom}{Isom}
\DeclareMathOperator{\length}{Length}

\DeclareMathOperator{\diam}{Diam}
\DeclareMathOperator{\epst}{Ep}
\DeclareMathOperator{\ep}{\Sigma}
\DeclareMathOperator{\tr}{Tr}
\DeclareMathOperator{\Hom}{Hom}

\renewcommand{\hat}{\widehat}

\newcommand{\pmat}[4]{\begin{pmatrix}{#1} & {#2}\\{#3} & {#4}\end{pmatrix}}
\newcommand{\smat}[4]{\left ( \begin{smallmatrix}{#1} & {#2}\\{#3} & {#4}\end{smallmatrix} \right )}

\newcommand{\param}{{\mathchoice{\mkern1mu\mbox{\raise2.2pt\hbox{$\centerdot$}}\mkern1mu}{\mkern1mu\mbox{\raise2.2pt\hbox{$\centerdot$}}\mkern1mu}{\mkern1.5mu\centerdot\mkern1.5mu}{\mkern1.5mu\centerdot\mkern1.5mu}}}

\numberwithin{equation}{section}

\theoremstyle{plain}
\newtheorem{thm}{Theorem}[section]
\newtheorem{cor}[thm]{Corollary}

\newtheorem{lem}[thm]{Lemma}

\newtheorem{bigthm}{Theorem}

\newtheorem{bigcor}{Corollary}

\newtheorem*{nonumthm}{Theorem}

\theoremstyle{definition}

\newtheorem{example}{Example}

\theoremstyle{definition}
\newtheorem*{remark}{Remark}
\newtheorem*{remarksenv}{Remarks}

\newenvironment{rmenumerate}{\begin{enumerate}}{\end{enumerate}}

\newenvironment{prooflist}{\begin{proof}\mbox{}\begin{list}{(\roman{enumi})}{\usecounter{enumi}\setlength{\labelwidth}{30pt}\setlength{\itemindent}{0pt}\setlength{\leftmargin}{20pt}\setlength{\itemsep}{5pt}}}{\end{list}\end{proof}}

\begin{document}

\title[Holonomy limits of $\CP^1$ structures]{Holonomy
  limits of\\complex projective structures}
\author{David Dumas}
\thanks{Research supported in part by the National Science Foundation.}

\address{
Department of Mathematics, Statistics, and Computer Science\\
University of Illinois at Chicago\\
\tt{ddumas@math.uic.edu}
}

\date{July 23, 2015 (revised).  May 25, 2011 (original).}


\maketitle

\section{Introduction}
\label{sec:introduction}

The set of complex projective structures on a compact Riemann surface
$X$ is parameterized by the vector space $Q(X)$ of holomorphic
quadratic differentials.  Each projective structure has an associated
holonomy representation, which defines a point in $\X(\Pi)$, the
$\SL_2(\C)$ character variety of the fundamental group $\Pi \define
\pi_1(X)$.  The resulting \emph{holonomy map} $\hol : Q(X) \to
\X(\Pi)$ is a proper holomorphic embedding.

In this paper we relate the large-scale behavior of the holonomy map
to the geometry of quadratic differentials on $X$.  In particular we
study the accumulation points of $\hol(Q(X))$ in the Morgan-Shalen
compactification of $\X(\Pi)$.  Such an investigation was proposed by
Gallo, Kapovich, and Marden in \cite[Sec.~12.4]{gkm}.

Boundary points in the Morgan-Shalen compactification are projective
equivalence classes $[\ell]$ of length functions $\ell : \Pi \to
\R^{+}$.  Each such function $\ell$ arises as the translation length
function of an isometric action of $\Pi$ on an $\R$-tree $T$; we say
such a tree $T$ \emph{represents} $[\ell]$.

Associated to each $\phi \in Q(X)$ there is an $\R$-tree $T_\phi$,
which is the space of leaves of the horizontal measured foliation of
$\phi$ lifted to the universal cover of $X$.  Our main result shows
that this tree predicts the Morgan-Shalen limit points of holonomy
representations associated to the ray $\R^+ \phi$, or more generally,
of any divergent sequence that converges to $\phi$ after rescaling.
More precisely, we show:

\begin{bigthm}
\label{thm:main-folding}
If $\phi_n \in Q(X)$ is a divergent sequence with projective limit
$\phi$, then any accumulation point of $\hol(\phi_n)$ in the
Morgan-Shalen boundary is represented by an $\R$-tree $T$ that admits
an equivariant, surjective straight map $T_\phi \to T$.
\end{bigthm}

The notion of a \emph{straight map} is discussed in Section
\ref{sec:straight}.  For the moment we simply note that such a map is
a morphism of $\R$-trees but it may not be an isometry because certain
kinds of folding are permitted.  For differentials with simple
zeros, however, we can rule out this behavior:

\begin{bigthm}
\label{thm:main-simple}
If $\phi_n \in Q(X)$ is a divergent sequence that converges
projectively to a quadratic differential $\phi$ with only simple
zeros, then $\hol(\phi_n)$ converges in the Morgan-Shalen
compactification to the length function associated with the dual tree
$T_\phi$.

In particular there is an open, dense, co-null subset of $Q(X)$
consisting of differentials $\phi$ for which $T_\phi$ is the unique
minimal limit action on an $\R$-tree arising from sequences with
$Q(X)$ with projective limit $\phi$.
\end{bigthm}

In order to pass from uniqueness of the limiting length function to a
unique limit action on a tree, the proof of this theorem uses a result of
Culler and Morgan \cite{culler-morgan}: The
tree representing a length function is determined up to equivariant
isometry except possibly when $\ell$ is an \emph{abelian length
  function} of the form $\ell(\gamma) = |\chi(\gamma)|$, where $\chi
:\Pi \to \R$ is a homomorphism.

For abelian length functions, the description of the isometry classes
of representing $\R$-trees is more complicated \cite{brown:bns}. 
However, in this case we can say more about the
corresponding quadratic differentials:

\begin{bigthm}
\label{thm:main-abelian}
Let $\chi : \Pi \to \R$ be a homomorphism.  If $\hol(\phi_n)$
converges in the Morgan-Shalen sense to the abelian length function
$|\chi|$, then the sequence $\phi_n$ converges projectively to
$\omega^2$, where $\omega \in \Omega(X)$ is a holomorphic $1$-form
whose imaginary part is the harmonic representative of $[\chi] \in
H^1(X,\R)$.
\end{bigthm}

We remark that the existence of sequences satisfying the hypotheses of
Theorem \ref{thm:main-abelian} is itself an open question.  Using the
results of \cite{gkm} one can construct a sequence of projective
structures on surfaces of a fixed
genus 
converging to an abelian length function, however it is not clear whether one
can also arrange for the underlying Riemann surface to remain
constant.
Naturally it would be interesting to
resolve this issue, ideally with either an explicit construction or a
geometrically meaningful obstruction to existence; we hope to return
to this in future work.

\smallskip

\subsection*{Rate of divergence}

A key step in understanding the limiting behavior of the holonomy
representations is to understand the \emph{rate} at which they diverge
as $\phi \to \infty$.  When equipped with the metric $|\phi|$, the
Riemann surface $X$ becomes a singular Euclidean surface whose
diameter is comparable to $\|\phi\|^{1/2}$.  We show that this is the
natural scale to use in understanding the action of $\hol(\phi)$ on
$\H^3$ by isometries:

\begin{bigthm}
\label{thm:main-proper}
The scale of the holonomy representation $\hol(\phi)$ is comparable to
$\|\phi\|^{1/2}$, i.e.
\begin{enumerate}
\item The translation length of any element of $\Pi$ in $\hol(\phi)$
acting on $\H^3$ is $O(\|\phi\|^{1/2})$, and
\item There is an element $\gamma \in \Pi$ whose translation length in
$\hol(\phi)$ is at least $c \|\phi\|^{1/2}$, where $c>0$ is a uniform
constant.
\end{enumerate}
\end{bigthm}

These statements are made precise in Theorem \ref{thm:length-bounds}
below.

Ultimately, our understanding of translation lengths of elements of
$\hol(\phi)$ acting on $\H^3$ comes from the construction of a
well-behaved map $\Tilde{X} \to \H^3$ that takes nonsingular
$|\phi|$-geodesics to nearly-geodesic paths in $\H^3$ parameterized
with nearly-constant speed.  (These maps are discussed in somewhat
more detail below, with the actual construction appearing in Section
\ref{sec:epstein}.)  When applied to a nonsingular $|\phi|$-geodesic
axis of an element $\gamma \in \Pi$, equivariance of the construction
shows that $|\phi|$-length of $\gamma$ in $X$, which is comparable to
$\|\phi\|^{1/2}$, is also comparable to the translation length in
$\H^3$.

Theorem \ref{thm:main-proper} gives another proof of the properness of
the holonomy map on $Q(X)$ (see also \cite[Thm.~11.4.1]{gkm},
\cite{tanigawa:divergence}) which is effective in the sense that it
includes an explicit growth estimate.  In Theorem
\ref{thm:effective-properness} we describe this \emph{effective
  properness} result in terms of an affine embedding of $\X(\Pi)$ and
an arbitrary norm on $Q(X)$.

\subsection*{Equivariant surfaces in $\H^3$}

The proofs of the main theorems are based on the analysis of surfaces
in hyperbolic $3$-space associated to complex projective structures.
The basic construction is due to Epstein \cite{epstein:envelopes}:
Starting from an open domain embedded in $\CP^1$ and a conformal
metric, one forms a surface in $\H^3$ from the envelope of a family of
horospheres.  The metric can be recovered from this surface by a
``visual extension'' procedure.

A natural generalization of this construction applies to a Riemann
surface that immerses (rather than embeds) in $\CP^1$ and a
conformal metric on the surface.  In our variant of Epstein's
construction, a single holomorphic quadratic differential $\phi \in
Q(X)$ provides \emph{both} of these data; the immersion is the
developing map of the projective structure with Schwarzian derivative $\phi$, and the
conformal metric is a multiple of the singular Euclidean metric
$|\phi|$.  The resulting \emph{Epstein-Schwarz map}
$\ep_\phi: \Tilde{X} \to \H^3$ is equivariant with respect to $\hol(\phi)$.

Using an explicit formula for the Epstein-Schwarz map we show that
when $\|\phi\|$ is large, this map shares key geometric properties
with the projection of $\Tilde{X}$ onto the dual tree $T_\phi$.
Namely, vertical trajectories of $\phi$ are mapped near geodesics in $\H^3$,
while compact segments on horizontal trajectories of $\phi$ are collapsed to
sets of small diameter.  These estimates are uniform outside a small
neighborhood of the zeros of $\phi$.

The main theorems are derived from these properties of Epstein-Schwarz
maps using a description of the Morgan-Shalen compactification in
terms of asymptotic cones of $\H^3$ (as in \cite{kapovich-leeb},
\cite{chiswell}).  We show that the sequence of Epstein-Schwarz maps
to $\H^3$ converge in a suitable sense to a limit map into an
$\R$-tree representing the limit of holonomy representations, and that
the local collapsing behavior described above leads to the global
straight map $T_\phi \to T$ from Theorem \ref{thm:main-folding}.

\subsection*{Comparison with other techniques}

The technique of relating the trajectory structure and Euclidean
geometry of a quadratic differential to the collapsing behavior of an
associated map has been used extensively in the study of harmonic maps
from hyperbolic surfaces to negatively curved spaces (including
$\H^2$, $\H^3$, and $\R$-trees), beginning with the work of Wolf
\cite{wolf:thesis} \cite{wolf:high-energy} on the Thurston
compactification of Teichm\"uller space.  More recently,
Daskalopoulos, Dostoglou, and Wentworth \cite{ddw:morgan-shalen}
studied the Morgan-Shalen compactification of the $\SL_2(\C)$ character
variety using harmonic maps, and our analysis of geometric limits of
Epstein-Schwarz maps follows a similar outline to their
investigation of equivariant harmonic maps to $\H^3$.

While harmonic maps techniques have been useful in the study of
complex projective structures (e.g.~\cite{tanigawa:grafting}
\cite{tanigawa:divergence} \cite{scannell-wolf:grafting}
\cite{dumas:antipodal}), for the purposes of Theorems
\ref{thm:main-folding}--\ref{thm:main-proper} the Epstein-Schwarz
maps have the advantage of a direct connection to the parameterization
of the space of projective structures by quadratic differentials.  In
addition, while harmonic maps are implicitly defined by minimization
of a functional (or solution of an associated PDE), the
Epstein-Schwarz map is given by an explicit formula which can be
analyzed directly, simplifying the derivation of our geometric
estimates.

\subsection*{Relating compactifications}
Our results show that it is natural to compare the compactification by
rays $\overline{Q(X)} = Q(X) \sqcup \PQ(X)$, where $\PQ(X) = (Q(X)
\setminus \{0\}) / \R^+$, with the closure of $\hol(Q(X))$ in the
Morgan-Shalen compactification $\overline{\X(\Pi)}$.  In terms of
these compactifications, Theorem \ref{thm:main-simple} can be
rephrased as

\addtocounter{bigcor}{\value{bigthm}}
\begin{bigcor}
\label{cor:main-compactification}
There is an open, dense, full-measure subset of $\partial
\overline{Q(X)}$ to which $\hol$ extends continuously as a map into
the Morgan-Shalen compactification $\overline{\X(\Pi)}$.  On this subset, the
extension of $\hol$ sends a ray $[\phi]$ of quadratic differentials to
the length function of the action of $\Pi$ on the dual tree of $\phi$.
\end{bigcor}

We also note that this extension is injective: A holomorphic quadratic
differential $\phi$ is determined by its horizontal measured
lamination $\lambda$ \cite{hubbard-masur}, and $\lambda$ is determined
by its intersection function $\left ( \gamma \mapsto i(\lambda,\gamma)
\right )_{\gamma \in \Pi}$, which is the length function of $\Pi$
acting on the dual tree of $\phi$.

While this gives a description of the limiting behavior of $\hol$ at
most boundary points, our results leave open the possibility that
there exist divergent sequences having the same projective limit in
$Q(X)$ but whose associated holonomy representations have distinct
limits in the Morgan-Shalen compactification of $\X(\Pi)$.  While we
suspect that this phenomenon occurs for some sequences (necessarily
converging to differentials with higher-order zeros), we do not know
of any explicit examples of this behavior.

\subsection*{Applications and related results}

The space $\ML(X)$ of measured geodesic laminations embeds in the
Morgan-Shalen boundary of $\X(\Pi)$, with image consisting of the
length functions associated to the trees $\{ T_\phi \: | \: \phi \in
Q(X)\}$.  In \cite{dumas-kent:dense}, Kent and the author showed that
the closure of $\hol(Q(X))$ in the Morgan-Shalen compactification
contains $\ML(X)$ by examining the countable subset of $Q(X)$ whose
associated holonomy representations are Fuchsian.  Theorem
\ref{thm:main-simple} (or Corollary \ref{cor:main-compactification})
gives an alternate proof of this result.

As in \cite{dumas-kent:dense}, our investigation of $\hol(Q(X))$ was
motivated in part by a connection to Thurston's skinning maps of
hyperbolic $3$-manifolds.  In \cite{dumas:finite}, the results of this
paper are used in the proof of:

\begin{nonumthm}
Skinning maps are finite-to-one.
\end{nonumthm}

Briefly, the connection between this result and holonomy of projective
structures is as follows: If the skinning map of a $3$-manifold $M$
with incompressible boundary $S$ had an infinite fiber, then there
would be a conformal structure $X$ on $S$ and an analytic curve
$\mathcal{C} \subset Q(X)$ consisting of projective structures whose
holonomy representations extend from $\pi_1(S)$ to $\pi_1(M)$.  This
extension condition constrains the limit points of $\hol(\mathcal{C})$
in the Morgan-Shalen compactification, and Theorem
\ref{thm:main-folding} is a key step in translating this into a
constraint on $\mathcal{C}$ itself.  Using analytic and symplectic
geometry in $Q(X)$, it is shown that these constraints are not
satisfied by any analytic curve, giving the desired contradiction.

\subsection*{Outline}

Section \ref{sec:background} contains background material on quadratic
differentials and projective structures, as well as some simple
estimates related to geodesics of quadratic differential metrics.

In Section \ref{sec:epstein} we introduce Epstein maps and specialize
to the case of interest, the Epstein-Schwarz map associated to a
quadratic differential.  The asymptotic behavior of sequences of such
maps is studied in Section \ref{sec:sequences}.

In Section \ref{sec:properness} we discuss the character variety and
then apply the estimates of the previous section to bound the size of
the holonomy representation.  We also give a new proof of the
properness of the holonomy map.

Finally, in Section \ref{sec:final} we discuss the Morgan-Shalen
compactification, dual trees of quadratic differentials, and straight
maps.  We then assemble the proofs of the main theorems from results
of sections \ref{sec:background}--\ref{sec:sequences}.

\subsection*{Acknowledgments}

The author thanks Peter Shalen and Richard Wentworth for helpful
conversations, and Richard Kent for asking interesting questions about
skinning maps that motivated some of this work.  The author also
thanks the anonymous referees for their careful reading and helpful
comments.

\section{Projective structures and quadratic differentials}
\label{sec:background}

\subsection{Projective structures}
Let $X$ be a compact Riemann surface of genus $g \geq 2$.  A (complex)
\emph{projective structure} on $X$ is a maximal atlas of conformal charts
mapping open sets in $X$ into $\CP^1$ whose transition functions are
the restrictions of M\"obius transformations.  Equivalently, a $\CP^1$
structure on $X$ can be specified by a locally injective holomorphic
map $f : \Tilde{X} \to \CP^1$, the \emph{developing map}, such that
for all $\gamma \in \Pi$ and $x \in \Tilde{X}$, we have
$$ f( \gamma \cdot x) = \rho(\gamma) \cdot f(x)$$
where $\rho : \Pi \to \PSL_2(\C)$ is a homomorphism, called the \emph{holonomy
  representation}.  The pair $(f,\rho)$ is uniquely determined by the
projective structure up to an element $A \in \PSL_2(\C)$, which acts by
$(f,\rho) \to (A \circ f, A \rho A^{-1})$.  For further discussion of
projective structures and their moduli see \cite[Ch.~7]{kapovich:book}
\cite{gunning:affine-projective} \cite{dumas:survey}.

While the holonomy representation naturally takes values in
$\PSL_2(\C)$, the representations that arise from projective
structures admit lifts to the covering group $\SL_2(\C)$
\cite[Sec.~1.3]{gkm}.  Furthermore, by choosing a spin structure on
$X$ it is possible to lift the holonomies of all projective structures
consistently (and continuously).  We will assume from now on that such a
structure has been fixed and so we consider only maps to $\SL_2(\C)$.

\subsection{Parameterization by quadratic differentials}
The space $P(X)$ of projective structures on $X$ is naturally an
affine space modeled on the vector space $Q(X)$ of holomorphic
quadratic differentials on $X$. The identification of the universal cover
$\Tilde{X}$ with the upper half-plane $\H$ induces the \emph{standard
  Fuchsian projective structure}, and this basepoint gives a
well-defined homeomorphism $P(X) \to Q(X)$.

This map sends a projective structure to the quadratic
differential $\phi \in Q(X)$ whose lift $\Tilde{\phi}$ to the
universal cover $\Tilde{X} \simeq \H$ satisfies
$$ \Tilde{\phi} = S(f) = \left ( \left(\frac{f''}{f'}\right)' -
  \frac{1}{2}\left(\frac{f''}{f'}\right)^2 \right )\:dz^2.$$
Here $S(f)$ is the \emph{Schwarzian derivative} of a developing
map $f$ of the projective structure.

\subsection{Developing a quadratic differential}

The inverse map $Q(X) \to P(X)$ can be
constructed as follows (following \cite[Ch.~2]{anderson:thesis}; see
also \cite{thurston:zippers}).  Given a quadratic differential $\phi \in
Q(X)$ with lift $\Tilde{\phi} \in Q(\H)$, we have the associated
$\sl_2(\C)$-valued holomorphic $1$-form
$$\omega_\phi = \frac{1}{2} \Tilde{\phi}(z)
\begin{pmatrix}
-z & 1\\-z^2 & z
\end{pmatrix}
dz.
$$
This form satisfies the structural equation $d \omega_\phi +
\frac{1}{2} [\omega_\phi,\omega_\phi] = 0$ because a Riemann surface
does not admit any nonzero holomorphic $2$-forms.  Thus there exists a
map $M_\phi : \Tilde{X} \to \SL_2(\C)$ whose Darboux derivative is
$\omega_\phi$ (see \cite[Thm.~7.14]{sharpe} for details), i.e. such that
$$\omega_\phi = M_\phi^{-1} dM_\phi.$$
This map is unique up to translation by an element of $\SL_2(\C)$.

The \emph{developing map} of $\phi$ is the holomorphic map $f_\phi :
\H \to \CP^1$ defined by
$$ f_\phi(z) = M_\phi(z) \cdot z$$
where in this expression $M_\phi(z)$ is considered as acting on
$\CP^1$ as a M\"obius transformation.  Of course the map $f_\phi$ is
only defined up to composition with a M\"obius map, but we speak
of \emph{the} developing map when the particular choice is not
important.

The map $f_\phi$ satisfies $S(f_\phi) = \Tilde{\phi}$ and is
equivariant with respect to the holonomy representation $\rho_\phi$
that is defined by the condition
$$ \rho_\phi(\gamma) M_\phi(z) = M_\phi(\gamma \cdot z) \rho_0(\gamma) $$
for all $\gamma \in \Pi$ and any $z \in \H$.  That
the choice of $z$ does not matter follows from the invariance of
$\omega_\phi$ under the action of $\Pi$ coming from the deck action on
$\Tilde{X}$ and the $\Ad \circ \rho_0$-action on $\sl_2(\C)$.  One can
think of $\rho_\phi(\gamma)$ as the ``non-abelian
period'' of the $1$-form $\omega_\phi$ along the loop $\gamma$ in $X$.

\subsection{Conformal and Riemannian metrics}
\label{subsec:conformal-metrics}

Given a Riemann surface $X$ with canonical line bundle $K$, a
\emph{conformal metric on $X$} is a continuous, nonnegative section
$\sigma$ of $K^{1/2} \tensor \overline{K^{1/2}}$ with
  the property that the function $ d_\sigma(x,y)= \inf_{\gamma:
    ([0,1],0,1) \to (X,x,y)} \int_\gamma \sigma$ defines a metric on
  $X$.  With respect to a local complex coordinate chart $z$ in which
  $\sigma$ is nonzero, we can write $\sigma = e^{\eta(z)} |dz|$ where
  $\eta$ is the \emph{log-density} of $\sigma$.  The metrics we
  consider will only vanish at finitely many points, and we extend
  $\eta$ to these points by defining $\eta(z) = -\infty$ if $\sigma(z)
  = 0$.

  A conformal metric is of class $C^k$ if it is nonzero and its
  log-density function in any chart is $k$ times continuously
  differentiable.  The Gaussian curvature of a $C^2$ conformal metric
  is given by
$$ K(z) = -4 e^{-2\eta} \eta_{z \bar{z}}$$
where subscripts denote differentiation.  Further discussion of
conformal metrics on Riemann surfaces can be found in \cite{huber:subharmonic}
\cite[Sec.~1.5, 4.1]{ahlfors:conformal-invariants}
\cite[Sec.~2.3]{jost:compact-riemann-surfaces}.

For any $\phi \in Q(X)$, the line element $|\phi|^{1/2}$ defines a
conformal metric on $X$ that is $C^\infty$ and flat ($K = 0$) away
from the set of zeros $Z_\phi = \phi^{-1}(0)$; this is a
\emph{quadratic differential metric}.  (See \cite[Ch.~III]{strebel:book} for
detailed discussion of such metrics.)  The total area is
$$\|\phi\|  = \int_X |\phi|,$$
which is the conformally natural $L^1$ norm on $Q(X)$.  A zero of $\phi$ of
order $k$ is a cone point of the metric $|\phi|^{1/2}$ with cone angle
$(k+2)\pi$.

For brevity we will sometimes refer to either the area form $|\phi|$
or the length element $|\phi|^{1/2}$ as the \emph{$\phi$-metric}.

\subsection{Quadratic differentials foliations and development}
\label{sec:qd-development}

Away from a zero of $\phi \in Q(X)$, there is always a local
\emph{natural coordinate}
$z$ such that $\phi = dz^2$.  Such a coordinate is unique up to
translation and $z \mapsto -z$.  Pulling back the lines in $\C$
parallel to $e^{i \theta} \R$ gives the \emph{foliation of angle
  $\theta$}, denoted $\F_\theta(\phi)$, which extends to a singular
foliation of $X$ with $(k+2)$-prong singularities at the zeros of
$\phi$ of order $k$.

The special cases $\theta =0,\pi/2$ are the
\emph{horizontal} and \emph{vertical} foliations, respectively.  We
sometimes abbreviate $\F(\phi) = \F_0(\phi)$. Each of these foliations
has a transverse measure coming from the natural coordinate
charts (e.g.~the vertical variation measure $|dy|$ for the horizontal
foliation).  Given a curve in $X$, we refer to its total measure with
respect to the horizontal foliation (resp.~vertical foliation) as its
\emph{height} (resp.~\emph{width}).

A path $\gamma : [0,1] \to X$ with interior disjoint from the zeros of
$\phi$ can be developed into $\C$ using local natural coordinate
charts.  The difference between the images of $\gamma(1)$ and
$\gamma(0)$ is the \emph{holonomy} of $\gamma$, which is well-defined up
to sign.  For example, the holonomy of a line segment with height $h$
and width $w$ is $\pm(w + i h)$.  Note that this holonomy construction
is equivalent to integrating the locally-defined $1$-form
$\sqrt{\phi}$; this should be contrasted with the integration of the
$1$-form $\omega_\phi$ used to construct the developing map $f_\phi$.
The interplay between these two integration constructions is an
underlying theme in our analysis of the Epstein-Schwarz map in
later sections.

\subsection{Quadratic differential geodesics}

Each free homotopy class of simple closed closed curves on $X$ can be
represented by a length-minimizing geodesic for the metric $|\phi|$,
which consists of a finite number of line segments joining zeros of
$\phi$.  The geodesic representative is unique unless it is a closed
leaf of $\F_\theta(\phi)$ for some $\theta \in S^1$, in which case
there is a cylinder foliated by parallel geodesic representatives.  In
the latter case we say the the geodesic is \emph{periodic}.

Similarly, any pair of points in $\Tilde{X}$ can be joined by a unique
geodesic segment for the lifted singular Euclidean metric
$|\Tilde{\phi}|$, which again consists of line segments
joining the zeros.  If such a geodesic segment does not contain any
zeros, it is \emph{nonsingular}.  Thus any geodesic segment in
$\Tilde{X}$ can be expressed as a union of nonsingular pieces.

We will need to extend some of these considerations to meromorphic
quadratic differentials with finitely many second-order poles.  With
respect to the singular Euclidean structure, each second-order pole
has a neighborhood that is isometric to a half-infinite cylinder.  If
$\phi$ has local expression $a z^{-2} + O(z^{-1})$ in a local
coordinate chart, then $a$ is the \emph{residue} of the pole and and
$2 \pi |a|$ is the circumference of the associated cylinder.  As in
the case of holomorphic differentials, an Euclidean line segment in $X$
(or its universal cover) is a length-minimizing geodesic.

Additional discussion of quadratic differential metrics and
geodesics can be found in \cite{strebel:book}
\cite[Sec.~4]{minsky:harmonic-maps}.

\subsection{Periodic geodesics}

Every quadratic differential metric has many periodic geodesics: Masur
showed that for any $\phi \in Q(X)$, there is a dense set of
directions $\theta \in S^1$ for which $\F_\theta(\phi)$ has a closed
leaf \cite{masur:closed}.  More generally, we have:

\begin{thm}[{Boshernitzan, Galperin, Kruger, and Troubetzkoy \cite{bgkt}}]
\label{thm:dense}
\mbox{}
For any $\phi \in Q(X)$, tangent vectors to periodic $\phi$-geodesics
are dense in the unit tangent bundle of $X$.
\end{thm}

Because a periodic geodesic for a quadratic differential metric always
sits in a parallel family foliating an annulus, any homotopy class
that can be represented by a periodic $\phi$-geodesic is also periodic
for all $\psi \in Q(X)$ sufficiently close to $\phi$.  Combining this
with the density of periodic directions, we have:

\begin{thm}
\label{thm:finite-periodic}
For any $\epsilon > 0$ there is a constant $w_0$ and finite set $P \subset \Pi$ such that for
any $\phi \in Q(X)$ with $\phi \neq 0$ there exists $\gamma \in P$
that is freely homotopic to a periodic $\phi$-geodesic that is nearly
vertical, i.e.~it is a leaf of $\F_\theta(\phi)$ for some $\theta \in
(\pi/2 - \epsilon, \pi/2 + \epsilon)$, and such that the flat annulus
foliated by parallels of the geodesic has width at least $w_0
\|\phi\|^{1/2}$.  The set $P$ can be taken to depend only on $X$ and
$\epsilon$.
\end{thm}

\begin{proof}
The statement is invariant under scaling so we can restrict attention
to the unit sphere in $Q(X)$.  By Theorem \ref{thm:dense} for each
such $\phi$ there exists a nearly-vertical periodic geodesic.  This
periodic geodesic persists (and remains nearly-vertical) in an open
neighborhood $U_\phi$ of $\phi$.  Shrinking $U_\phi$ if necessary we
can also assume that the width of the flat annulus is bounded below
throughout $U_\phi$. The unit sphere in $Q(X)$ is compact
so it has a finite cover by these sets.  Choosing a representative in
$\Pi$ for the periodic curve in each element of the cover gives the
desired set $P$, and taking the minimum of the width of the annuli
over these sets gives $w_0$.
\end{proof}

Further discussion of periodic trajectories for quadratic differential
metrics can be found in \cite[Sec.~4]{masur-tabachnikov}.

\subsection{Comparing geodesic segments}

If two quadratic differentials are close, then away
from the zeros, a geodesic segment for one of them is nearly geodesic
for the other.  We make this idea precise in the following lemmas,
which are used in Section \ref{sec:sequences}.

Note that throughout this section, the holonomy of a path refers to
the Euclidean development of a quadratic differential as defined in
Section \ref{sec:qd-development} above.

\begin{lem}
\label{lem:delta}
Let $U
\subset \C$ be an open set and $\psi = \psi(z) dz^2$ a holomorphic
quadratic
differential on $U$ satisfying
\begin{equation*}
|\psi(z) - 1| < \delta < \frac{1}{2}.
\end{equation*}
If $J$ is a line segment in $U$ with holonomy $z_J$ with respect to
$dz^2$, then the holonomy $w_J$ of $J$ with respect to $\psi$
satisfies $| z_J - w_J | < \delta |z_J|,$ and in particular $w_J \neq
0$ if $z_J \neq 0$.
\end{lem}

\begin{proof}
By hypothesis the function $\psi(z)$ does not have zeros in $U$, so there is a
unique branch of $\psi(z)^{1/2}$ with positive real part, which
satisfies
$$ |\psi(z)^{1/2} - 1| < \delta.$$
Here we have used that $\delta < \frac{1}{2}$ to ensure that $\psi(z) \mapsto
\psi(z)^{1/2}$ is contracting.  Since holonomy is obtained by
integrating $\psi(z)^{1/2}$, the inequality above gives
$$ |z_J - w_J| \leq \int_J |\psi(z)^{1/2} - 1| |dz| < \delta |z_J|.$$
\end{proof}

\begin{lem}
\label{lem:phipsi}
Let $\phi \in Q(X)$ be a holomorphic quadratic differential and $U
\subset Q(X)$ an open, contractible, $\phi$-convex set that does not
contain any zeros of $\phi$.  If $\psi$ is a holomorphic quadratic
differential on $U$ satisfying
$$ \frac{|\psi - \phi|}{|\phi|} < \delta < \frac{1}{4},$$
then:
\begin{rmenumerate}
\item Any natural coordinate for $\psi$ is injective on $U$.
\item For any $p,q \in U$ we have
$$ d_\phi(p,q) > 4/5 d_\psi(p,q).$$
\end{rmenumerate}
Furthermore, if $J$ is a $\phi$-geodesic segment in $U$ of length $L$
that is
  not too close to $\partial U$, i.e.
$$ d_\phi(J,\partial U) > 4 \delta L,$$
then we also have:
\begin{rmenumerate}
\setcounter{enumi}{2}
\item The endpoints of 
  $J$ are joined by a nonsingular $\psi$-geodesic segment $J' \subset
  U$,
\item The segment $J'$ satisfies $d_{\psi}(J',\partial U) >
  \frac{1}{4} d_{\phi}(J,\partial U)$ and $d_{\phi}(J',\partial U) >
  \frac{1}{4} d_{\phi}(J,\partial U)$.
\item The width $w'$, height $h'$, and length $L'$ of $J'$ with
  respect to $\psi$ satisfy
$$ \max(|L'-L|, |w'-w|,|h'-h|) < \delta L,$$
where $L$,$w$, and $h$ are the corresponding quantities for $J$ with
respect to $\phi$.
\end{rmenumerate}
\end{lem}

\begin{proof}
Identify $U$ with its image by a natural coordinate $z$ for $\phi$. 
Then
$\psi = \psi(z) dz^2$ satisfies $|\psi(z) - 1| < \delta$.  
Now we
repeatedly apply the holonomy estimate from Lemma \ref{lem:delta}.

(i)\: By Lemma \ref{lem:delta}, any line segment in $U$ has nonzero
$\psi$-holonomy and $U$ is convex, so $U$ develops injectively by a
natural coordinate $\zeta$ for $\psi$.

(ii)\: Again using Lemma \ref{lem:delta} we have 
\begin{equation}
\label{eqn:close-holonomy}
|( z(p) - z(q) ) - ( \zeta(p) - \zeta(q) ) | < \delta |z(p) - z(q) |
\end{equation}
from which it follows that $|z(p) - z(q)| > (1+\delta)^{-1}|\zeta(p) -
\zeta(q)|$.  Convexity implies that $d_\phi(p,q) = |z(p) - z(q)|$
while the injectivity of $\zeta$ on $U$ gives $d_\psi(p,q) \leq
|\zeta(p)-\zeta(q)|$.  Noting that $(1+\delta)^{-1} > 4/5$ gives the
desired estimate.

(iii)\: Equation \eqref{eqn:close-holonomy} also gives the bound
$$|\zeta(p) - \zeta(q)| > (1 - \delta) |z(p) - z(q)|$$
however we can only equate the left hand side with the distance
$d_\psi(p,q)$ in cases where $p$ and $q$ are joined by a $\psi$-segment
in
$U$.  However, since $U$ injects into the $\zeta$-plane, the minimum
distance from $J$ to $\partial U$ is realized by such a segment, and
we have
$$ d_\psi(J,\partial U) > (1-\delta) d_\phi(J,\partial U) >
4\delta(1-\delta)L.$$

Let $\{j_0,j_1\}$ denote the endpoints of $J$ and translate the
coordinates $z$ and $\zeta$ so that $z(j_0) = \zeta(j_0) = 0$.
Parameterize $J$ by $\alpha(t)$ so that $z(\alpha(t)) = t z(j_1)$.
Then any point on $\zeta(J)$ has the form $\zeta(\alpha(t))$, while a
point on the segment $I$ in $\C$ joining $\zeta(j_0)$ to $\zeta(j_1)$
has the
form $t \zeta(j_1)$ for $t \in [0,1]$.  We
estimate
$$ | t \zeta(j_1) - \zeta(\alpha(t)) | 
\leq |t \zeta(j_1) - t z(j_1) | + |z(\alpha(t)) - \zeta(\alpha(t))|.$$
Each term on the right is the difference in $\phi$- and
$\psi$-holonomy vectors of a path of $\phi$-length at most $L$ (with a
coefficient of $t$ in the first term).  By Lemma \ref{lem:delta} each
term is at most $\delta L$, so the segment $I$ lies in a $2 \delta
L$-neighborhood of $\zeta(J)$.  

Since $2 \delta < 4\delta(1-\delta)$, we have $I \subset \zeta(U)$
and $J' = \zeta^{-1}(I)$ defines a nonsingular $\psi$-geodesic segment.

(iv)\: From (iii) we have $d_\psi(J',\partial U) >
(1-\delta)d_\phi(J,\partial U)
- 2 \delta L$, and by hypothesis $2 \delta L <
(1/2)d_\phi(J,\partial U)$.  Combining these and using
$(1/2 - \delta) > 1/4$ gives the desired estimate.

(v)\:The $\phi$-holonomy of $J$ is $w + i h$, while the
$\psi$-holonomy is $w' + i h'$.  The comparison of these quantities
therefore follows immediately from the holonomy estimate.
\end{proof}

\subsection{The Schwarzian derivative of a conformal metric}

Given two conformal metrics $\sigma_i= e^{\eta_i} |dz|$, $i=1,2$, the
\emph{Schwarzian derivative} of $\sigma_2$ relative to $\sigma_1$ is
the quadratic differential
\begin{equation}
\label{eqn:schwarzian-def}
\B(\sigma_1,\sigma_2) = \left [ (\eta_2)_{zz} - (\eta_2)_z^2 -
  (\eta_1)_{zz} + (\eta_1)_z^2 \right ] \: dz^2.
\end{equation}
Note that this differential is not necessarily holomorphic.  This
generalization of the classical Schwarzian derivative was introduced
by Osgood and Stowe \cite{osgood-stowe:schwarzian} (though in their
construction the result is a symmetric real tensor which has the
expression above as the $(2,0)$ part).  The classical Schwarzian
derivative can be recovered from the metric version as follows: If $f
: \Omega \to \C$ is a locally injective holomorphic function on a
domain $\Omega$, then
\begin{equation*}
\label{eqn:schwarzian-pullback}
S(f) = 2 \B(|dz|, f^*(|dz|)).
\end{equation*}
We will use the following properties of the generalized Schwarzian
derivative, each of which follows easily from the formula above.
\begin{enumerate}
\item[(B1)] \textbf{Cocycle:} For any triple of conformal
metrics
  $(\sigma_1,\sigma_2,\sigma_3)$ on a given domain, we have
\begin{equation*}
\label{eqn:schwarzian-cocycle}
\B(\sigma_1,\sigma_3) = \B(\sigma_1,\sigma_2) +
\B(\sigma_2,\sigma_3).
\end{equation*}
\item[(B2)] \textbf{Naturality:} If $f : \Omega \to \Omega'$ is a
conformal map
  of domains in $\C$ (or $\CP^1$), and $(\sigma_1,\sigma_2)$ are
  metrics on $\Omega'$, then we have
\begin{equation*}
\B(f^* \sigma_1, f^* \sigma_2) = f^* \B(\sigma_1,\sigma_2)
\end{equation*}
\item[(B3)] \textbf{Flatness:} If a conformal metric $\sigma_0$ on a
domain in $\C$ satisfies
  $\B(|dz|,\sigma_0)=0$, then there exist $k > 0$ and $A \in
  \SL_2(\C)$ such that $k A^*\sigma_0$ is the restriction of one of the
  following metrics:
\begin{enumerate}
\item The hyperbolic metric $2(1-|z|^2)^{-1}|dz|$ on $\Delta$.
\item The Euclidean metric $|dz|$ on $\C$.
\item The spherical metric $2(1+|z|^2)^{-1}$ on $\CP^1$.
\end{enumerate}
\end{enumerate}
The metrics described in (B3) will be called \emph{M\"obius
  flat}.  It follows from (B1) that the Schwarzian derivative of
a metric $\sigma = e^\eta |dz|$ relative to $|dz|$ is equal to its
Schwarzian derivative relative to any M\"obius flat metric
$\sigma_{\mathrm{flat}}$, and is
given by 
\begin{equation}
\label{eqn:short-schwarzian-formula}
  \B(|dz|, e^\eta |dz|) = \B(\sigma_{\mathrm{flat}}, e^\eta |dz|) =
\left (
\eta_{zz} - (\eta_z)^2\right) dz^2.
\end{equation}
We also note that property (B2) implies that the Schwarzian is
well-defined for
pairs of conformal metrics on a Riemann surface.

\begin{lem}
\label{lem:holomorphic}
Let $\sigma$ be a conformal metric of constant curvature.  Then
the differential $\B(\sigma,\sigma')$ is holomorphic if and only
if the curvature of $\sigma'$ is also constant.
\end{lem}

\begin{proof}
An elementary calculation using \eqref{eqn:schwarzian-def} gives
$$ \bar{\partial} \B(\sigma,\sigma') = K_z \sigma^2 -
K'_z \sigma'^2,$$
where $K$ (respectively $K'$) is the Gaussian curvature function of
$\sigma$ (resp.~$\sigma'$).  By hypothesis $K_z \equiv 0$, and
$\sigma'^2$ is
a nondegenerate area form, so the expression above vanishes if and
only if $K'$ is constant.
\end{proof}

\subsection{Metrics associated to a $\CP^1$ structure}
\label{sec:metrics}

As before let $X$ be a compact Riemann surface and let $(f,\rho)$ be a
projective structure on $X$ with Schwarzian $\phi \in Q(X)$.
Associated to these data are three conformal metrics:
\begin{itemize}
\item The hyperbolic metric $\hyp$ on $X$,
\item The singular Euclidean metric $|\phi|^{1/2}$, and
\item The pullback metric $f^*\sph$ on $\Tilde{X}$, where $\sph$ is a
  spherical metric on $\CP^1$.  
\end{itemize}

Taking pairs of these metrics gives three associated
Schwarzian derivatives, which by Lemma \ref{lem:holomorphic} are
holomorphic except possibly at the zeros of $\phi$.  By (B2) we
have:
\begin{equation*}
\phi = 2 \B(\hyp,f^*\sph),
\end{equation*}
and for the other pairs we introduce the notation
\begin{equation*}
\begin{split}
\hat{\phi} &= 2 \B(|\phi|^{1/2},f^*\sph),\\
\flat &= 2\B(\hyp,|\phi|^{1/2}).
\end{split}
\end{equation*}
Note that $f^*\sph$ is actually a metric on the universal cover rather
than on $X$ itself.  However, by (B2) its Schwarzian relative to any
$\Pi$-invariant metric is a $\Pi$-invariant quadratic differential, so
in the expressions above we have implicitly identified this
differential with the one it induces on $X$.

By (B1) the differentials $\phi$,$\hat{\phi}$,$\flat$ have
a linear relationship:
\begin{equation}
\label{eqn:linear}
\hat{\phi} = \phi - \flat.
\end{equation}

Near a zero of $\phi$, one can choose coordinates so that $\phi = z^k
dz^2$.  Calculating in these coordinates and using the explicit
expression for $\B(\param,\param)$, it is easy to check that
$\flat$ extends to a meromorphic differential on $X$ with poles
of order $2$ at the zeros of $\phi$.  At a zero of $\phi$ of order
$k$, the residue of $\flat$ is $-\frac{k(k+4)}{8}$.
Of course by \eqref{eqn:linear}, the differential $\hat{\phi}$ also
has poles of order $2$ at the zeros of $\phi$ and is holomorphic
elsewhere.

We will be interested in comparing the geometry of $\phi$ and
$\hat{\phi}$ when $\phi$ is ``large''.  Note that $\flat$ is
independent of scaling and $Q(X)$ is finite-dimensional, so $(\phi -
\hat{\phi}) = \flat$ ranges over a compact set of meromorphic
differentials.  Thus when $\phi$ has large norm, one expects
$|\flat/\phi|$ to be small and for $\phi$ and $\hat{\phi}$ to be
nearly the same away from the zeros of $\phi$.  Quantifying this in
terms of the geometry of $|\phi|$, we have:

\begin{lem}[Bounding $\flat$]
\label{lem:boundingbeta}
For any $\phi \in Q(X)$ we have
\begin{equation}
\label{eqn:betabound}
\left | \frac{\flat(z)}{\phi(z)} \right | \leq \frac{6}{d(z)^2},
\end{equation}
where $d(z)$ is the $\phi$-distance from $z$ to $Z_\phi$.  Furthermore,
if $\grad$ denotes the gradient with respect to the metric
$|\phi|^{1/2}$, then we also have
\begin{equation}
\label{eqn:betaderivbound}
\left | \grad \left ( \frac{\flat(z)}{\phi(z)}\right ) \right | \leq
\frac{48}{d(z)^3}.
\end{equation}
\end{lem}

\begin{proof}
  We work in a natural coordinate $z$ for $\phi$ and use this
  coordinate to identify differentials with holomorphic functions,
  i.e.~$\flat(z) / \phi(z)$ becomes $\flat(z)$.

  By applying a translation it suffices to consider the point $z=0$.
  By definition of the function $d(z)$, we can also assume that the
  $z$-coordinate neighborhood contains an open Euclidean disk $D$ of
  radius $d = d(0)$ centered at $0$.  If $d$ is greater than the
  $\phi$-injectivity radius of $X$, we work in the universal cover but
  suppress this distinction in our notation.

  Let $h : D \to \H$ be a developing map for the hyperbolic metric of
  $X$ restricted to $D$, so $\flat = S(h)$.  Since
  $h$ is a univalent map on $D$, the Nehari-Kraus theorem gives
$|S(h)(0)| \leq
  6/d^2$, which is \eqref{eqn:betabound}.

  Since $\flat(z)$ is holomorphic and we are working in the natural
coordinate for $\phi$, the gradient is given by $|\nabla \flat(z)| =
|\flat'(z)|$.  The estimate \eqref{eqn:betaderivbound} then follows
immediately from the Cauchy integral formula applied to a circle of
  radius $d(z)/2$ centered at $z$. 
\end{proof}

\section{Epstein maps}
\label{sec:epstein}

In this section we review a construction of C.~Epstein (from the
unpublished paper \cite{epstein:envelopes}) which produces surfaces in
hyperbolic space from domains in $\CP^1$ equipped with conformal
metrics.  We analyze the local geometry of these surfaces, first for
general conformal metrics and then for the special case of a quadratic
differential metric.  While at several points we mention results and
constructions from \cite{epstein:envelopes}, our treatment is
self-contained in that we provide proofs of the properties of these
surfaces that are used in the sequel.

\subsection{The construction.}
For each $p \in \H^3$, following geodesic rays from $p$ out to the
sphere at infinity $\partial_\infty \H^3 \simeq \CP^1$ defines a
diffeomorphism $U_p \H^3 \to \CP^1$, where $U \H^3$ denotes the unit
tangent bundle of $\H^3$.  Let $V_p$ denote pushforward of the metric
on $U_p \H^3$ by this map, which we call the \emph{visual metric}
from $p$.  For example, in the unit ball model of $\H^3$, the
visual metric from the origin is the usual spherical metric of $S^2
\simeq \CP^1$.

\begin{thm}[{Epstein \cite{epstein:envelopes}}]
\label{thm:epstein-existence}
Let $X$ be a Riemann surface equipped with a $C^1$ conformal metric
$\sigma$, and let $f : X \to \CP^1$ be a locally injective
holomorphic map.  Then there is a unique continuous map $\epst(f,\sigma)
:
X \to \H^3$ such that for all $z \in X$, we have
$$(f^*V_{\epst(f,\sigma)(z)})(z) = \sigma(z).$$
Furthermore, the point $\epst(f,\sigma)(z)$ depends only on the $1$-jet
of
$\sigma$ at $z$, and if $\sigma$ is $C^k$, then $\epst(f,\sigma)$ is
$C^{k-1}$.
\end{thm}

We call $\epst(f,\sigma)$ the \emph{Epstein map} associated to
$(X,f,\sigma)$, and sometimes refer to its image as an \emph{Epstein
  surface}.  However, note that $\epst(f,\sigma)$ is not necessarily an
immersion, and could even be a constant map (e.g.~if $\sigma =
f^*(V_p)$).

The Epstein map has a natural lift $\Hat{\epst}(f,\sigma) : X \to U
\H^3$ as
follows.  For $p \in \H^3$ and $x \in \CP^1$, let $v_{p \to x}$ denote
the unit tangent vector to the geodesic ray from $p$ that has ideal
endpoint $x$.  We define
$$ \hat{\epst}(f,\sigma)(z) = \left ( \epst(f,\sigma)(z),
v_{\epst(f,\sigma)(z) \to f(z)}
\right ).$$
Clearly $\pi \circ \hat{\epst}(f,\sigma) = \epst(f,\sigma)$, where $\pi
: U \H^3
\to \H^3$ is the projection.  Furthermore, since $f$ is locally
injective, the same is true of $\hat{\epst(f,\sigma)}$.

Epstein also shows that if $\epst(f,\sigma)$ is an immersion in a
neighborhood of $z$, then there is a neighborhood $U$ of $z$ such that
$\epst(f,\sigma)(U)$ is a convex embedded surface in $\H^3$, and
$\hat{\epst}(f,\sigma)(U)$ is its set of unit normal vectors.

\subsection{Explicit formula.}

An explicit formula for $\epst(f,\sigma)$ is given in the unit ball
model of $\H^3$ in \cite{epstein:envelopes}.  We will now describe the
same map in model-independent terms.  Since the construction is local
and equivariant with respect to M\"obius transformations, it suffices
to consider the case of a $C^1$ conformal metric $\sigma = e^\eta
|dz|$ on an open set $\Omega \subset \C$ (an affine chart of $\CP^1$),
and to determine a formula for the Epstein map of $(\Omega, \sigma,
\id)$.  In what follows we write $\epst$ for $\epst(f,\sigma)$, with
the dependence on $\sigma$ (and its log-density $\eta$) being
implicit.

Define a map $\Tilde{\epst} : \Omega \to \SL_2(\C)$ by
\begin{equation}
\begin{split}
\label{eqn:epstein-formula}
\Tilde{\epst} (z) &=\pmat{e^{-\eta / 2}(1 + z \eta_z)}{e^{\eta/2}
  z}{e^{-\eta/2}\eta_z}{e^{\eta/2}}\\
&= \pmat{1}{z}{0}{1}
\pmat{1}{0}{\eta_z}{1}
\pmat{e^{-\eta / 2}}{0}{0}{e^{\eta/2}}
\end{split}
\end{equation}
where subscripts denote differentiation, and we have written
$\eta$ instead of $\eta(z)$ for brevity.

Our choice of an affine chart $\C \subset \CP^1 \simeq \partial_\infty
\H^3$ distinguishes the ideal points $0,\infty$ and the geodesic
joining them.  Let $P_0 \in \H^3$ denote the point on this geodesic so
that the visual metric $V_{P_0}$ and the Euclidean metric $|dz|$
induce the same norm on the tangent space at $0$.  (In the standard
upper half-space model of $\H^3$, we have $P_0 = (0,0,2)$.)

The Epstein map of $\sigma$ is the $P_0$-orbit map of $\Tilde{\epst}$,
i.e.
$$\epst(z) = \Tilde{\epst}(z) \cdot P_0.$$
Similarly, the lift $\hat{\epst}(z)$ is the orbit map of the unit vector
$v_{P_0 \to 0} \in U_{P_0}\H^3$.

This description of $\epst(z)$ can be derived from Epstein's formula
(\cite[Eqn.~2.4]{epstein:envelopes}) by a straightforward calculation,
or the visual metric property of Theorem \ref{thm:epstein-existence}
can be checked directly.  However, since we will not use the visual
metric property directly, we take \eqref{eqn:epstein-formula} as
the definition of the Epstein map.  This formula will be used in all
subsequent calculations.

Recall that the unit tangent bundle of a Riemannian manifold has a
canonical contact structure, and lifting a co-oriented locally convex
hypersurface by its unit normal field gives a Legendrian submanifold.
The following property of Epstein maps shows that $\hat{\epst}$ can be
seen as providing a unit normal vector for $\epst$, even at points
where the latter is not an immersion.

\begin{lem}
\label{lem:legendrian}
The map $\hat{\epst}$ is a Legendrian immersion into $U\H^3$.
\end{lem}

\begin{proof}
As before we work locally, in a domain $\Omega \subset \C$. Using $v =
v_{P_0 \to 0}$ as a basepoint, the $\SL_2(\C)$-action identifies the
unit tangent bundle of $\H^3$ with the homogeneous space $\SL_2(\C) /
A$ where $A = \stab(v) = \left \{ \smat{e^{i \theta}}{0}{0}{e^{-i
    \theta}} \right \}$.  Let $\mathfrak{g}$ denote $\sl_2(\C) =
\Lie(\SL_2(\C))$ and $\mathfrak{a} \define \Lie(A)$.

The set of Killing vector fields on $\H^3$ (i.e.~elements of
$\mathfrak{g}$) that are orthogonal to $v$ at $P_0$ descends to a
codimension-$1$ subspace of $\mathfrak{g}/\mathfrak{a} \simeq T_{P_0}
U \H^3$, and the corresponding $\SL_2(\C)$-equivariant distribution on
$T U\H^3$ is the contact structure.  In coordinates, this
orthogonality condition determines the subspace $\{ \smat{a}{b}{c}{-a}
\: | \: \Re(a) = 0 \} \subset \mathfrak{g}.$

Therefore, to check that $\hat{\epst}$ is Legendrian it suffices to
show that the (Darboux) derivative $\tilde{\epst}^{-1} d\tilde{\epst}
: T \Omega \to \mathfrak{g}$ takes values in this space.
Differentiating formula \ref{eqn:epstein-formula} gives an expression
of the form
$$\Tilde{\epst}^{-1} d\Tilde{\epst} = \frac{1}{2}\pmat{i \left ( \eta_x
dy
  - \eta_y dx \right )}{e^\eta (dx + i dy)}{*}{-i \left ( \eta_x dy -
  \eta_y dx \right )}$$ where $z = x + i y$.  Since the upper-left
entry is purely imaginary, the map $\Hat{\epst}$ is tangent to the
contact distribution.  Since the upper-right entry is injective (as a
linear map $T_z \Omega \to \C$), the map is an immersion and thus
Legendrian.
\end{proof}

\subsection{First derivative and first fundamental form}
In this section we assume that the conformal metric $\sigma$ is $C^2$.
Using the formula \eqref{eqn:epstein-formula} and the expression for
the hyperbolic metric in the homogeneous model $\H^3 \simeq \SL_2(\C)
/ \SU(2)$, it is straightforward to calculate the first fundamental
form $\I$ of the Epstein surface.  In complex coordinates, the result
is:
\begin{equation*}
\begin{split}
\I &=\;\; (\eta_{zz} - \eta_{z}^2)(1 + 4 e^{-2 \eta} \eta_{z
    \bar{z}}) dz^2\\
&\;\;+ 
\left (4 e^{-2 \eta} |\eta_{zz} - \eta_{z}^2|^2 + \tfrac{1}{4} e^{2
    \eta} (1 + 4
e^{-2 \eta} \eta_{z \bar{z}})^2\right ) dz d\bar{z}\\
&\;\;+
(\eta_{\bar{z}\bar{z}} - \eta_{\bar{z}}^2)(1 + 4 e^{-2 \eta} \eta_{z
    \bar{z}}) d\bar{z}^2\\
\end{split}
\end{equation*}
Notice that $(\eta_{zz} - \eta_{z}^2)$ represents the Schwarzian
$\B(\sph,\sigma)$ of the metric $\sigma = e^\eta |dz|$,
where $\sph$ denotes a M\"obius flat metric on $\CP^1$ (see
\eqref{eqn:short-schwarzian-formula}).  Recall that the Gaussian
curvature of the metric $\sigma$ is $ K = - 4e^{-2 \eta} \eta_{z
  \bar{z}}$.  In terms of these quantities, we have
\begin{equation}
\label{eqn:firstform}
\I
=  \frac{4}{\sigma^2} |\B(\sph,\sigma)|^2 +
  \frac{1}{4}(1-K)^2 \sigma^2 + 2(1-K) \Re(\B(\sph,\sigma))
\end{equation}

\subsection{Second fundamental form and parallel flow.}
The Epstein surface for the metric $e^t \sigma$ (with log-density
$\eta + t$) is the result of applying the time-$t$ normal flow to the
surface for $\sigma$ itself.  In such a parallel flow, the first
fundamental form evolves according to $ \dot{\I} = -2 \II$ where $\II$
is the second fundamental form.  (Here and below we use the notation
$\dot{x}$ for $\left . \frac{dx}{dt} \right |_{t=0}$.)

In order to simplify the expressions for these derivatives we
work in a local conformal coordinate $z$ and introduce the $1$-forms:
\begin{equation*}
\begin{split}
\theta &= e^{\eta + t} dz\\
\chi &= \frac{2}{\theta}\B(\sph,\sigma) +  \frac{\bar{\theta}}{2}(1 -
K)
\end{split}
\end{equation*}
Note that $\theta$ is $(1,0)$ form of unit norm with respect to $e^t
\sigma$.  In terms of these quantities we can rewrite
\eqref{eqn:firstform} as $\I = \chi \bar{\chi}$, and so we have $\II =
-\Re(\dot{\chi} \bar{\chi})$.  Since $\dot{\theta} = \theta$, $\dot{K} =
-2K$, and $\tfrac{d}{dt}\B(\sph,e^t\sigma) =
0$, the $1$-form $\chi$ satisfies
\begin{equation*}
\begin{split}
-\dot{\chi} &=
\frac{2}{\theta}\B(\sph,\sigma) + \frac{\bar{\theta}}{2}(1+K)
 \end{split}
 \end{equation*}
Substituting, we obtain
\begin{equation}
\label{eqn:secondform}
\II
= \frac{4}{\sigma^2} |\B(\sph,\sigma)|^2 -
  \frac{1}{4}(1 - K^2) \sigma^2   - 2 K \re(\B(\sph,\sigma))
\end{equation}

\subsection{The Epstein-Schwarz map}

Let $X$ be a compact Riemann surface and $\phi \in Q(X)$ a quadratic
differential.  In this section we will often need to work on the
surface $X' = X \setminus Z_\phi$ obtained by removing the zeros of
$\phi$.  Let $(f,\rho)$ denote the developing map and holonomy
representation of the projective structure on $X$ satisfying $S(f) =
\phi$.

The developing map $f$ and the conformal metric $|2 \phi|^{1/2}$ on $X'$
induce an Epstein map
$$ \ep_\phi \define \epst (f, {|2 \phi|^{1/2}}) : \Tilde{X'} \to \H^3$$
which we call the \emph{Epstein-Schwarz map}.  Similarly, we have the
lift $\Hat{\ep}_\phi : \Tilde{X'} \to U\H^3$ to the unit tangent
bundle.  Note that $\Tilde{X'}$ denotes the complement of
$\Tilde{Z}_\phi = \Tilde{\phi}^{-1}(0)$ in $\Tilde{X}$, rather than
the universal cover of $X'$ itself.  The factor of $\sqrt{2}$ in the
definition of $\ep_\phi$ arises naturally when considering the first
and second fundamental forms of the image surface (e.g.~Lemma
\ref{lem:speedcurvature} and Example \ref{ex:dzsquared} below).

Recall from Section \ref{sec:metrics} that associated to $\phi = 2
\B(\hyp,\sph)$ we have the meromorphic differentials $\hat{\phi} = 2
\B(|\phi|^{1/2},f^*\sph)$ and $\flat = 2\B(\hyp,|\phi|^{1/2})$.  We now
calculate the first and second fundamental forms of the
Epstein-Schwarz map in terms of these quantities.

\begin{lem}[Calculating $\I$ and $\II$]
\label{lem:firstsecond}
The first fundamental form of the Epstein-Schwarz map $\ep_\phi$
is
\begin{equation}
\label{eqn:es-firstform}
\I =  \frac{|\hat{\phi}|^2 + |\phi|^2}{2 |\phi|} 
- \re \hat{\phi}
\end{equation}
This map is an immersion at $x$ if and only if
$$|\hat{\phi}(x)| \neq |\phi(x)|,$$
and at any such point, the second fundamental form is
\begin{equation}
\label{eqn:es-secondform}
\II = \frac{|\hat{\phi}|^2 - |\phi|^2}{2 |\phi|}
\end{equation}
Furthermore, using the unit normal lift $\Hat{\ep}$ to define the
derivative of the unit normal at points where $\ep$ is not an
immersion, the formula for above extends to all of $X'$.
\end{lem}

\begin{proof}
Substituting $K = 0$, $\B(\sph,\sigma) = -\tfrac{1}{2}
\hat{\phi}$, and $\sigma^2 = 2 |\phi|$ into
\eqref{eqn:firstform}-\eqref{eqn:secondform} gives
the formulas for $\I$ and $\II$,  so we need only determine where
$\ep_\phi$ is an immersion and
justify that the formula for $\II$ holds even when it is not.

The $1$-form $\chi$ defined above reduces to 
$$ \chi = \frac{1}{\sqrt{2}}
\left ( \frac{\hat{\phi}}{\phi^{1/2}} + \bar{\phi}^{1/2} \right ),$$
where $\phi^{1/2}$ is a locally-defined square root of $\phi$.  The
Epstein map fails to be an immersion when the first fundamental form
$\I = \chi \bar{\chi}$ is degenerate, i.e.~when $\chi$ and
$\bar{\chi}$ are proportional by a complex constant of unit modulus.
By the expression above this occurs when $\hat{\phi}/ \phi^{1/2} = a
\phi^{1/2}$ for some $a \in \C$ with $|a| = 1$.  This is equivalent to
$|\hat{\phi}| = |\phi|$.

Finally, in calculating the second fundamental form above, we used the
equation $\dot{\I} = -2 \II$ for the normal flow of an immersed
surface.  The same formula holds for the flow associated to an
immersed Legendrian surface in $U \H^3$, so by Lemma
\ref{lem:legendrian} it applies to Epstein lift $\hat{\ep}_\phi$.
Thus, formula \eqref{eqn:secondform} gives the second fundamental form
of $\ep_\phi$ in this generalized sense.
\end{proof}

We see from this lemma that the pullback metric $\I$ is \emph{not}
compatible
with the conformal structure of the Riemann surface $X$; its $(2,0)$
part $\hat{\phi}$ represents the failure of $\ep$ to be a conformal
mapping onto its image.  On the other hand, $\II$ is a quadratic form
of type $(1,1)$ and so it induces a metric compatible with $X$.

We will now use these expressions for the fundamental forms of the
Epstein surface to derive estimates based on the relative difference
between the differentials $\hat{\phi}$ and $\phi$.  

More precisely bounds will
be based on the function $\eps : X' \to \R$ give by
$$\eps(x) = \left | \frac{\flat(x)}{\phi(x)} \right |,$$
for which we already have some estimates by Lemma
\ref{lem:boundingbeta}.

\begin{lem}
\label{lem:speedcurvature}
The first and second fundamental forms $\I,\II$ of the Epstein-Schwarz
map $\ep = \ep_\phi$ satisfy the following:
\begin{rmenumerate}
\setlength{\leftmargin}{0in}
\item The principal directions of the quadratic form $\I$, relative to
  a background metric on $X$ compatible with its conformal structure,
  are given by the horizontal and vertical directions of the quadratic
  differential $\hat{\phi}$.  Here the horizontal direction
  corresponds to the maximum of $\I$ on a unit circle in a tangent
  space.

\item The images of the horizontal and vertical foliations of
  $\hat{\phi}$ are the lines of curvature of the Epstein surface.

\item Let $\xi_h$ and $\xi_v$ denote unit horizontal and vertical
  vectors for $|\hat{\phi}|$ at $x \in X'$.  If $\eps(x) <
\tfrac{1}{2}$, then
  the images of these vectors satisfy
  \begin{equation*}
\begin{split}
  \| \ep_*(\xi_h) \| & <   \eps(x)\\
  \sqrt{2} < \| \ep_*(\xi_v)\| & < \sqrt{2} + \eps(x).
\end{split}
\end{equation*}

\item Let $\kappa_h, \kappa_v$ denote the
  principal curvatures of $\ep$ associated to the
  horizontal and vertical directions of $\hat{\phi}$ at $x$,
respectively.
  If $\eps(x) < \tfrac{1}{2}$, then
\begin{equation*}
\begin{split}
|\kappa_h| & > \frac{1}{\eps(x)}\\
|\kappa_v| & = \frac{1}{\kappa_h} < \eps(x)
\end{split}
\end{equation*}
\end{rmenumerate}
\end{lem}

\begin{remark}
Parts of this lemma could also be derived from results in
\cite{epstein:envelopes}:
\begin{enumerate}
\item Epstein shows that the vertical and horizontal foliations of
$(\eta_{zz} - \eta_z^2)dz^2$ are mapped to lines of curvature by the
Epstein map of $e^\eta |dz|$.  This includes part (ii) of the lemma
above as a special case.
\item Epstein also relates the curvature $2$-forms of the conformal
metric and of the first fundamental form of the associated Epstein
surface; in the case of a flat metric this implies that the principal
curvatures satisfy $\kappa_1 \kappa_2 = 1$.
\end{enumerate}
\end{remark}

\begin{prooflist}
\item Since the principal directions are orthogonal, it suffices to
  consider one of them.  By \eqref{eqn:es-firstform}, the only part of
  $\I$ that varies on a conformal circle in a tangent space of $X$ is
  the term $\re \hat{\phi}(v)$.  Thus the norm is maximized for
  vectors such that $\hat{\phi}(v)$ is real and positive, which is
  equivalent to $v$ being tangent to the horizontal foliation of
  $\hat{\phi}$, as desired.

\item Since $\II$ is real and has type $(1,1)$, the eigenspaces of the
  shape operator $\I^{-1} \II$ are the principal directions of the
  quadratic form $\I$, which by (i) are the vertical and horizontal
  directions of $\hat{\phi}$.  Thus any vertical or horizontal leaf of
  $\hat{\phi}$ is a line of curvature.

\item First of all, it will be convenient to estimate
  $|\flat/\hat{\phi}|$, using the hypothesis that $\eps(x) <
  \tfrac{1}{2}$:
$$
\left |\frac{\flat}{\hat{\phi}}\right | = \frac{|\flat|}{|\phi - \flat|}
\leq \frac{|\flat|}{|\phi| - |\flat|} <
\frac{|\flat|}{\tfrac{1}{2}|\phi|}
= 2 \eps(x).$$

Let $n_h = \| \ep_* \xi_h\|$ and $n_v = \|
  \ep_* \xi_v \|$.  
Using formula \eqref{eqn:es-firstform} and the fact that $\xi_h(x)$ and
$\xi_v(x)$ are unit with respect to $|\hat{\phi}|$, we calculate
\begin{equation}
\label{eqn:nhnv}
\begin{split}
n_h^2  &= \I(\xi_h) = \frac{\left(|\hat{\phi}| - |\phi|\right )^2}{2
  |\phi \hat{\phi}|}\\
n_v^2 &= \I(\xi_v) = \frac{\left(|\hat{\phi}| +
    |\phi|\right )^2}{2 |\phi \hat{\phi}|} = 2 + n_h^2
\end{split}
\end{equation}
Since $\hat{\phi} = \phi - \flat$, we have
$| |\hat{\phi}| - |\phi|| \leq |\flat|$.  Substituting into the expression for
$n_h^2$ gives
$$ 
n_h^2 \leq \frac{|\flat|^2}{2|\phi \hat{\phi}|} <
\eps(x)^2$$
and the estimate on $n_h$ follows.  By the last equality of
\eqref{eqn:nhnv} we have
$$
\sqrt{2} < n_v = \sqrt{2 + n_h^2} < \sqrt{2} + n_h < \sqrt{2} + \eps(x)$$
as required.

\item Using \eqref{eqn:es-firstform}-\eqref{eqn:es-secondform} we find
  $\kappa_h \kappa_v = \det(\I^{-1} \II) = 1$, so we need only estimate
$\kappa_v$.
By (ii) the curvatures are obtained by multiplying the eigenvalues of
$\I^{-1}$ 
  by $\frac{|\hat{\phi}|^2 - |\phi|^2}{2 |\phi|}$, and we have
\begin{equation}
\label{eqn:kappav}
\kappa_v = \frac{|\hat{\phi}|^2 - |\phi|^2}{(|\hat{\phi}| + |\phi|)^2}
= \frac{|\hat{\phi}| - |\phi|}{|\hat{\phi}| + |\phi|}
\end{equation}
As before we use $||\hat{\phi}| -|\phi|| \leq |\flat|$, giving
$$|\kappa_v| \leq \frac{|\flat|}{|\phi| + |\hat{\phi}|} \leq \eps(x).$$
\end{prooflist}

\begin{lem}[Curvature of vertical leaves]
\label{lem:leafcurvature}
Let $\hat{L}$ denote a leaf of the vertical foliation of $\hat{\phi}$,
parameterized by $|\hat{\phi}|$-length, and for any $x \in \hat{L}$
let $k(x)$ denote the curvature of $\ep_\phi(\hat{L})$ at
$\ep_\phi(x)$.  Let $d(x)$ denote the $\phi$-distance from $x$ to
$Z_\phi$.  Then for any $x$ such that $d(x) > 2 \sqrt{3}$ we have
$$k(x) <  15d(x)^{-2}.$$
\end{lem}

\begin{proof}
All estimates in this proof involve tensors evaluated at a single
point $x \in \hat{L}$, so we abbreviate $d = d(x)$, $\phi = \phi(x)$,
etc..  By Lemma \ref{lem:boundingbeta} the hypothesis $d > 2\sqrt{3}$
gives $\eps = |\flat / \phi| < 1/2$ and $1/2 < |\hat{\phi}/ \phi| < 3/2$.
In particular this means that the estimates of Lemma
\ref{lem:speedcurvature} apply.

The image of $\hat{L}$ is a line of curvature of $\ep_\phi$
corresponding to the principal curvature $\kappa_v$.  Splitting
the curvature of its image in $\H^3$ into tangential and normal
components, we have have $k^2 = \kappa_g^2 + \kappa_v^2$
where $\kappa_g$ is the geodesic curvature of $\hat{L}$ at $x$
with respect to $\I$.  By Lemma \ref{lem:speedcurvature} we have
$$ \kappa_v <  |\flat / \phi | < 6 d^{-2}.$$

Let $\xi_h,\xi_v$ denote unit vertical and horizontal vectors of
$\hat{\phi}$ at $x$, which are tangent to the principal curvature
directions.  As in the proof of Lemma \ref{lem:speedcurvature}, we
denote by $n_h$ the norm of $\xi_h$ with respect to $\I$.  By an
elementary calculation in Riemannian geometry, the geodesic curvature
of a line of curvature satisfies
\begin{equation}
\label{eqn:kg-formula}
|\kappa_g| = \frac{| \xi_h(\kappa_v) |}{n_h \: |\kappa_v - \kappa_h|},
\end{equation}
where $\xi_h(\kappa_v)$ denotes the derivative of the function
$\kappa_v$ with
respect to the vector $\xi_h$.

Applying \eqref{eqn:kappav}, we have
$$ \kappa_h - \kappa_v =  \frac{1}{\kappa_v} - \kappa_v =
\frac{4|\phi \hat{\phi}|\;\;\;\;}{|\hat{\phi}|^2-|\phi|^2},$$
and similarly for the derivative,
\begin{equation*}
\begin{split}
\xi_h(\kappa_v) &= \xi_h \left ( \frac{|\hat{\phi}| -
|\phi|}{|\hat{\phi}| +
  |\phi|} \right ) = \xi_h \left ( 1 - \frac{2}{|\hat{\phi}/\phi| + 1}
\right ) \\
& =
2 \left ( \frac{|\phi|}{|\hat{\phi}| + |\phi|} \right )^2 \xi_h \left
( | \hat{\phi} / \phi | \right )
\end{split}
\end{equation*}
Since $|\hat{\phi} / \phi| = | 1 - \flat/\phi|$, we have
$|\xi_h(|\hat{\phi} / \phi|)| \leq  |\xi_h(\flat / \phi)|$.  
Using the bound on the norm of the $|\phi|$-gradient of $\flat/\phi$
from Lemma \ref{lem:boundingbeta} and the fact that the $\phi$-norm of
$\xi_h$ is $|\phi / \hat{\phi}|^{1/2}$,  we obtain
$$ | \xi_h(\flat / \phi) | \leq 48 d^{-3} |\phi / \hat{\phi}|^{1/2}.$$

Recall from \eqref{eqn:nhnv} that $n_h^2 =  (|\phi| - |\hat{\phi}|)^2
/ (2|\phi \hat{\phi}|)$.  Substituting these expressions
into \eqref{eqn:kg-formula} and simplifying gives
\begin{equation*}
|\kappa_g(x)| = 
\frac{|\phi| + |\hat{\phi}|}{2 \sqrt{2} \, | \phi \hat{\phi}|^{1/2}}
\xi_h(|\hat{\phi}/\phi|) 
<  24 \sqrt{2} \left [ \frac{ |\phi|^2}{|\hat{\phi}| ( |\phi| +
  |\hat{\phi}|) } \right ] d^{-3}
\end{equation*}
Since $1/2 < |\hat{\phi}/\phi| < 3/2$ it follows that the 
bracketed expression is bounded by $4/3$, so finally we have
$$ |\kappa_g(x)| < 32 \sqrt{2}\, d^{-3}.$$
Returning to the curvature function $k$, we combine the bounds for
$\kappa_g$ and $\kappa_v$ above and use $d > 2 \sqrt{3}$ to obtain
$$ k = (\kappa_g^2 + \kappa_v^2)^{1/2} < \left ( (32 \sqrt{2} \, d^{-3})^2 + (6
d^{-2})^2 \right) ^{1/2} < \left ( \frac{512}{3} + 36 \right
)^{1/2}d^{-2} <
15 d^{-2}.$$
\end{proof}

Next we combine the above results concerning the derivative and
the curvature of the Epstein-Schwarz map to estimate lengths of images
of segments.  The following theorem is the only result in this section
which is used in the sequel.

\begin{thm}[Collapsing]
\label{thm:maingeom}
  There exist $D_0 > 0$ and $C_0 > 0$ such that for all $\phi \in
  Q(X)$ we have
\begin{rmenumerate}
\item For any $d > D_0$, the restriction of $\ep_\phi$ to $X
  \setminus N_d(Z_\phi)$ is locally $(\sqrt{2} + C_0 d^{-2})$-Lipschitz with
  respect to the $\phi$-metric.
\item If $[y_1,y_2]$ is a segment on a vertical leaf of
  $\hat{\phi}$ and $d = d_{\phi}([y_1,y_2], Z_\phi) > D_0$, then
$$ ( \sqrt{2} - C_0 d^{-2}) \: d_{\phi}(y_1,y_2) < d_{\H^3}(\ep_\phi(y_1),
\ep_\phi(y_2)) < ( \sqrt{2} + C_0 d^{-2}) \: d_{\phi}(y_1,y_2).$$
\item If $[x_1,x_2]$ is a segment on a horizontal leaf of $\hat{\phi}$
  and $d = d_{\phi}([x_1,x_2],Z_\phi) > D_0$, then
$$ d_{\H^3}(\ep_\phi(y_1), \ep_\phi(y_2)) < C_0 d^{-2}
d_{\phi}(x_1,x_2).$$
\end{rmenumerate}
\end{thm}

\begin{proof}
The proof will show that one can take $D_0 = 4$ and $C_0 = 28$.

Since vertical segments are geodesics in the $\hat{\phi}$-metric, the
upper bound from (ii) follows from (i).

We first consider upper bounds on distances.  We can integrate a bound
on the derivative of $\ep_\phi$ over a path to obtain an upper
bound on the length of the image, and thus on the distance between
endpoints.  Since $d > D_0 > 2 \sqrt{3}$ we can apply the derivative
estimates
from Lemma
\ref{lem:speedcurvature} and combining them with 
Lemma \ref{lem:boundingbeta} we obtain
$$
d_{\H^3}(\ep_\phi(z_1),
\ep_\phi(z_2)) < (\sqrt{2} + 6 d^{-2}) \: d_{\phi}(z_1,z_2)
$$
for any $z_1,z_2$ that are joined by a minimizing geodesic in $X
\setminus N_d(Z_\phi)$.  This implies (i) and, since vertical segments
are minimizing geodesics, the upper bound from (ii).   For a
horizontal segment $[x_1,x_2]$, these lemmas give
$$
d_{\H^3}(\ep_\phi(x_1),
\ep_\phi(x_2)) < 6 d^{-2} \: d_{\phi}(x_1,x_2)\
$$
and (iii) follows.

To complete the lower bound for case (ii), we note that the lower
bound on the derivative of $\ep_\phi$ in the vertical direction
from Lemma \ref{lem:speedcurvature} implies
$$ \length(\ep_\phi([y_1,y_2])) > (\sqrt{2} - 6 d^{-2}) d_\phi(y_1,y_2).$$
Recall that a path in $\H^3$ with curvature bounded above by $k < 1$
and parameterized by arc length is $1/\sqrt{1 - k^2}$-bi-Lipschitz
embedded (see e.g.~\cite[App.~A]{leininger}).  Since $d > D_0 = 4$,
Lemma \ref{lem:leafcurvature} implies that the image of a vertical
leaf segment has curvature $k < 15 d^{-2} < 1$.  Combining this with
the length estimate and using that $\sqrt{1 - k^2} \geq 1 - k$ for
$k < 1$, we obtain
\begin{equation*}
\begin{split}
 d_{\H^3}(\ep_\phi(y_1),\ep_\phi(y_2)) &> (1 - 15d^{-2})(\sqrt{2} - 6
d^{-2}) d_\phi(y_1,y_2)\\
&> (\sqrt{2} - 28 d^{-2}) d_\phi(y_1,y_2)
\end{split}
\end{equation*}
completing the proof of (ii).
\end{proof}

\subsection{Quasigeodesics}

Let $I$ denote a closed interval, half-line, or $\R$.  Recall that a
parameterized path $\gamma : I \to M$ in a metric space $M$ is a
$(K,C)$-quasigeodesic if for all $a,b \in I$ we have
$$
K^{-1} |b-a|  - C \leq d(\gamma(a),\gamma(b)) \leq K |b-a| + C.
$$

The following property of quasigeodesics in $\H^3$ is well-known (see
e.g.~\cite{kapovich:book} \cite[Sec~11.8]{ratcliffe}).

\begin{lem}
\label{lem:stability}
For all $K \geq 1$ and $C \geq 0$ there exists $L = L(K,C) \geq 0$ with the
following property: If $\gamma : I \to \H^3$ is a
$(K,C)$-quasigeodesic, and if $J$ is
the geodesic segment in $\H^3$ with the same endpoints as $\gamma(I)$,
then the Hausdorff distance between $J$ and $\gamma(I)$ is at most
$L(K,C)$. \flushright \qedsymbol
\end{lem}

The lemma applies to quasigeodesic rays and lines, where the
``endpoints'' of $\gamma(I)$ and $J$ are allowed to lie on the sphere
at infinity.

We will also want to recognize quasigeodesics using the local
criterion provided by the following lemma.

\begin{lem}
\label{lem:local-to-global}
For all $K\geq1$ and $C\geq0$ there exist $R(K,C) > 0$, $K'(K,C) \geq 1$, and
$C'(K,C) \geq 0$ with following property: If $\gamma : I \to \H^3$ is a
$(K,C)$-quasigeodesic when restricted to each interval of length
$R(K,C)$, then $\gamma$ is a $(K'(K,C)),C'(K,C))$-quasigeodesic
(globally).  Furthermore, these quantities can be chosen to satisfy
\begin{equation}
\label{eqn:kprimecprime}
\left .
\def\arraystretch{1.2}%
\begin{array}{ll}
K'(K,C) \to 1\\
C'(K,C) \to 0\\
R(K,C) \: \text{bounded}
\end{array}
\right \}
  \text{ as } (K,C) \to (1,0).
  \end{equation}
\end{lem}

Without the claim about limits of $K',C',R$, this lemma represents a
well-known property of quasigeodesics in $\H^3$ (and more generally,
in $\delta$-hyperbolic metric spaces).  Proofs can be found in
\cite[Sec.~7]{gromov:essays} \cite[Sec.~3.1]{cdp}.  We therefore
concern ourselves with the limiting behavior of $K',C',R$ as $(K,C)
\to (1,0)$.

\begin{proof}[Proof of \eqref{eqn:kprimecprime}.]
A $(K,C)$-quasigeodesic is also a
$(1+\epsilon,\epsilon)$-quasigeodesic for some $\epsilon$ that goes
to zero as $(K,C) \to (1,0)$.  We assume from now on that $\gamma : I
\to \H^3$ is a $(1+\epsilon,\epsilon)$-quasigeodesic on segments of
length $1$ (i.e.~we set $R = 1$).  We will compare $\gamma(I)$ to the
piecewise geodesic path formed by the images of regularly spaced
points in $I$; in order to obtain good estimates we will need for the
spacing of these points will be much larger than $\epsilon$, but to
still go to zero as $\epsilon \to 0$.

Consider the triangle in $\H^3$ formed by $a = \gamma(t)$, $b =
\gamma(t+\epsilon^{1/8})$, $c = \gamma(t+2\epsilon^{1/8})$ for some
$t$ such that $[t,t+2\epsilon^{1/8}] \subset I$.  Assume that
$2\epsilon^{1/8} < 1$, so the path $\gamma$ is a
$(1+\epsilon,\epsilon)$-quasigeodesic on $[t,t+2\epsilon^{1/8}]$ and
we have 
\begin{equation*}
\begin{split}
d(a,b), d(b,c) &\in [\epsilon^{\frac{1}{8}}(1+\epsilon)^{-1}-\epsilon,\:
\epsilon^{\frac{1}{8}}(1+\epsilon) +\epsilon],\\
d(a,c) &\in
[2\epsilon^{\frac{1}{8}}(1+\epsilon)^{-1}-\epsilon,
\:2\epsilon^{\frac{1}{8}}(1+\epsilon) + \epsilon].
\end{split}
\end{equation*}
For small $\epsilon$ it follows that $d(a,c) \approx d(a,b) + d(a,c)$
and the triangle is nearly degenerate; a calculation using the
hyperbolic law of cosines shows that such a hyperbolic triangle has interior angle
at $b$ satisfying $\theta > \pi - 5 \epsilon^{7/16}$.

Let $V = \{ k \in \epsilon^{1/8}\Z \: | \:(k \pm \epsilon^{1/8}) \in I
\}$, and consider the path in $\H^3$ obtained by joining successive
elements of $\gamma(V)$ by geodesic segments.  By the estimates above
this piecewise geodesic path has segments of length at least $s$ and
angles between adjacent segments greater than $\pi - \delta$, where
\begin{equation*}
\begin{split}
s &= \epsilon^{1/8}/2 < (\epsilon^{1/8}(1+\epsilon)^{-1} - \epsilon)\\
\delta &= 5 \epsilon^{7/16}
\end{split}
\end{equation*}
By \cite[Thm.~I.4.2.10]{notes-on-notes}, if we have $ s \sin(s -
\delta) > \delta $ then such a piecewise geodesic path is
$k$-bi-Lipschitz embedded for $k = \cos(s)$.  For the values given
above we find $s \sin(s - \delta) \sim \frac{1}{4} \epsilon^{1/4}$ as
$\epsilon \to 0$.  Comparing exponents (i.e.~$1/4 < 7/16$) we find
that the condition is satisfied for $\epsilon$ sufficiently small.
Thus the path is bi-Lipschitz embedded with
$$k = \cos(\epsilon^{1/8}/2).$$
Note that $k \to 1$ as $\epsilon \to 0$.

For any $p,q \in I$ there exist $p',q' \in V$ with $|p-p'|,|q-q'| <
2\epsilon^{1/8}$.  Using the $k$-Lipschitz property of $\gamma(V')$ and
the fact
that $\gamma$ is $(1+\epsilon,\epsilon)$-quasigeodesic on segments of
length $2 \epsilon^{1/8}$, we have
\begin{equation*}
 k^{-1}(|p-q|-4\epsilon^{\frac{1}{8}}) - 4\epsilon^{\frac{1}{8}}(1 +
\epsilon) - 2\epsilon \leq
d(\gamma(p), \gamma(q))
\leq \;\;  k(|p-q|+4\epsilon^{\frac{1}{8}}) +
4\epsilon^{\frac{1}{8}}(1 + \epsilon) + 2 \epsilon.
\end{equation*}
Thus we take $K' = k$ and $C' = 4\epsilon^{1/8}(k + 1+\epsilon) + 2
\epsilon$, and \eqref{eqn:kprimecprime} follows.
\end{proof}

\subsection{Height and distance}
So far our analysis of the Epstein-Schwarz map has focused on
leaves of the foliations of the quadratic differential $\hat{\phi}$,
which is the sum of the Schwarzian of the projective structure
(i.e.~$\phi$) and a correction term ($-\flat$).  We will now use
Theorem \ref{thm:maingeom}, Lemma \ref{lem:boundingbeta}, and the
quasigeodesic estimates of the previous section to study the
restriction of $\ep_\phi$ to a geodesic of the $\phi$-metric.

The following theorem shows that the height of a $\phi$-geodesic
segment provides a good estimate for the distance between the
endpoints of its image by $\ep_\phi$, as long as the segment is far
from $Z_\phi$.

\begin{thm}
\label{thm:height-estimate}
There exists $M>0$ and decreasing functions $K'(m) > 1$, $C'(m) > 0$
defined for $m > M$ with the following property: Let $\phi \in Q(X)$
and let $J = [x_0,x_1]$ be a nonsingular and non-horizontal
$\phi$-geodesic segment in $\Tilde{X}$ with height $h$ and length $L$.
If $d = d_\phi(J,Z_\phi) > m(1 + \sqrt{L})$ for some $m \geq M$, then
  \begin{equation}
  \label{eqn:distance-estimate}
K'(m)^{-1} \sqrt{2} \, h - C'(m) < d_{\H^3}(\ep_\phi(x_0),\ep_\phi(x_1)) <
  K'(m) \sqrt{2} \, h + C'(m).
  \end{equation}
  Furthermore, we have $(K'(m),C'(m)) \to (1,0)$ as $m \to \infty$.
\end{thm}

Because of this constant factor of $\sqrt{2}$ in the estimate above,
it will be helpful in the proof and subsequent discussion to sometimes
use the following terminology: Suppose $J$ is a nonsingular
$\phi$-geodesic segment.  We say that a parameterization $J(t)$ of
this segment is a \emph{parameterization by $2\phi$-height} if the
$\phi$-height of $J([t_1,t_2])$ is equal to
$\tfrac{1}{\sqrt{2}} |t_2-t_1|$ for all $t_1,t_2$ in the domain.  This
is equivalent to the condition that the height of $J([t_1,t_2])$ with
respect to the quadratic differential $2 \phi$ is $|t_2-t_1|$.  Of
course any non-horizontal segment has a parameterization by
$2\phi$-height.  Theorem \ref{thm:height-estimate} shows that under
the stated hypotheses, the Epstein-Schwarz map $\ep_\phi$ sends a
segment parameterized by $2\phi$-height to a path in $\H^3$ that is
$(K,C)$-quasi-geodesic with $K \approx 1$ and $C \approx 0$ when $m$
is large.

\begin{proof}
We will make several assumptions of the form $m > c$, where $c$ is a
constant.  At the end we take $M$ to be the supremum of these
constants.

Let $U$ denote the $d/2$-neighborhood of $J$ with respect to the
$\phi$-metric, so $d_\phi(U,Z_\phi) = d/2 > m/2$. Define
$$ \delta = \sup_U \frac{|\hat{\phi} - \phi|}{|\phi|} = \sup_U
\frac{|\flat|}{|\phi|}.$$
Using the bound on $|\flat|/|\phi|$ from Lemma \ref{lem:boundingbeta}
and $d > m$ we obtain
$$ \delta < \frac{24}{m^2}.$$
and similarly, using $d > m \sqrt{L}$, we have
$$ 4 \delta L < \frac{96 L}{d^2} < \frac{96}{m^2}.$$
We now assume $m > 16$ which by the above estimates is more than
sufficient to ensure $\delta < 1/4$ and $d_\phi(J, \partial U) = d/2 > 4
\delta L$, so Lemma \ref{lem:phipsi} applies to $U$ and any subsegment
$J_1$ of $J$.  In particular $U$ contains a nonsingular
$\hat{\phi}$-geodesic segment $\hat{J}_1$ with the same endpoints as
$J_1$ and which satisfies
\begin{equation*}
\begin{split} 
&d_{\phi}(\hat{J}_1,\partial U) >
\frac{d}{8}, \;\; \text{and}\\
&\max(|\hat{h_1}-h_1|,|\hat{w_1}-w_1|) < \delta L < \frac{24}{d^2} < 1,
\end{split}
\end{equation*}
where $h_1,w_1$ are the $\phi$-height and width of $J_1$, and and
$\hat{h}_1,\hat{w}_1$ are the $\hat{\phi}$-height and width of
$\hat{J}_1$.

Now suppose that the subsegment $J_1$ has height at most $d/16$.  Then
$\hat{J}_1$ has height bounded by $1 + d/16 < d/8 <
d_\phi(\hat{J}_1,\partial U)$ and there are
$\hat{\phi}$-vertical and horizontal geodesic segments contained in $U$
that together with $\hat{J}_1$ form a right triangle $\hat{T}$ that
lies in $U$.  Taking $m > 2 D_0$ we can apply Theorem \ref{thm:maingeom} to the
vertical and
horizontal sides of $\hat{T}$ in order to estimate the distance
between the $\ep_{\phi}$-images of the endpoints $\{y_0,y_1\}$ of $J_1$,
obtaining
\begin{equation*}
\begin{split}
\left ( \sqrt{2} - \frac{4C_0}{d^2} \right ) \hat{h}_1 - \frac{4C_0}{d^2}
\hat{w_1}  \;\;\;\; &\leq\\ d_{\H^3}(\ep_\phi(y_0)&,\ep_\phi(y_1))\\
&\leq \;\;\;\; \left ( \sqrt{2} + \frac{4C_0}{d^2} \right ) \hat{h}_1 +
\frac{4C_0}{d^2}
\hat{w_1}
\end{split}
\end{equation*}
Using $|\hat{h_1}-h_1| < 24/d^2$ and $\hat{w}_1 < L + 24/d^2$, and the
fact that these estimates can be applied to any subsegment of $J_1$,
we find that $\ep_\phi$ maps the parameterization of $J_1$ by
$2\phi$-height to a $(K,C)$-quasigeodesic path in $\H^3$ with
\begin{equation}
\label{eqn:kc-for-height}
K = \left ( 1 - \frac{4C_0}{\sqrt{2} \, d^2}\right )^{-1}, \;\;\;\;
C =  \frac{ 4 C_0 (L+ 48d^{-2}) + 24\sqrt{2}}{d^2}.
\end{equation}
From these expressions it is clear that for $m$ large enough, the
assumptions $d>m$ and $d > m \sqrt{L}$ give upper bounds for $K,C$,
and that these decrease toward $1,0$, respectively, as $m \to \infty$.
Since this quasigeodesic property holds on each subsegment of $J$
whose height is at most $d/16 > m/16$, by
taking $m$ large enough we can apply Lemma \ref{lem:local-to-global},
and the parameterization of $J$ by $2\phi$-height is mapped by $\ep_\phi$ to a
$(K'(m),C'(m))$-quasigeodesic path in $\H^3$, where $K'(m) \to 1$ and
$C'(m) \to 0$ as $m \to \infty$.  The estimate
\eqref{eqn:distance-estimate} for
$d_{\H^3}(\ep_\phi(x_0),\ep_\phi(x_1))$ follows.
\end{proof}

For horizontal segments, the above proof applies up to
\eqref{eqn:kc-for-height}, and we can take $J_1 = J$ since the
condition that $h_1 < d/16$ is vacuous.  We conclude:

\begin{cor}[of proof]
\label{cor:horizontal-height-estimate}
There exist $M>0$ and a decreasing function $C'(m) > 0$
defined for $m > M$ with the following property: Let $\phi \in Q(X)$
and let $J = [x_0,x_1]$ be a nonsingular horizontal
$\phi$-geodesic segment in $\Tilde{X}$ with length $L$.
If $d = d_\phi(J,Z_\phi) > m(1 + \sqrt{L})$ for some $m \geq M$, then
\begin{equation}
  \label{eqn:horizontal-distance-estimate}
  \length(\ep_{\phi}(J)) \leq C'(m)
  \end{equation}
  Furthermore, we have $C'(m) \to 0$ as $m \to \infty$. \flushright
\qedsymbol
\end{cor}

\subsection{Examples}

Because the construction of the Epstein-Schwarz map associated to a
quadratic differential is purely local, we can consider the behavior
for some simple differentials on the complex plane to illustrate the
geometric properties studied in Theorems \ref{thm:maingeom} and
\ref{thm:height-estimate}.

\begin{example}
\label{ex:dzsquared}
Consider the quadratic differential $\phi = dz^2$ on $\C$.  Note that
$|\phi|^{1/2} = |dz|^{1/2}$ is M\"obius flat, so $\flat=0$, $\hat{\phi} = \phi$,
and Theorem \ref{thm:maingeom} applies to the trajectories of $\phi$.

The covering map $f : \C \to \C^*$ given by $f(z) = \exp(i \sqrt{2}
z)$ satisfies $S(f) = \phi$, so we can use this as a model for the
associated developing map.  The metric $\sqrt{2} |\phi|^{1/2}$ on $\C$
pushes forward to the metric $|dz|/|z|$ on $\C^*$.  In the standard
unit ball model of $\H^3$, this metric agrees with the spherical
metric on the equator, so the image of the equator by the Epstein map
is the origin.  Invariance of $|dz|/|z|$ under the action of $\R^+$ by
dilation then shows that the full Epstein map of this metric on $\C^*$ is
the orthogonal projection of $\partial_\infty \H^3$ onto the geodesic
$g_{0,\infty}$ joining the ideal points $0,\infty$.

Therefore the Epstein-Schwarz map $\ep_\phi : \C \to \H^3$ is the
composition of $f$ with this projection, or equivalently, $\ep_\phi(z)
= g(\Im(z))$ where $g(t)$ is an arc length parameterization of
$g_{0,\infty}$.  We see the behavior predicted by letting $d \to
\infty$ in the estimates of Theorem \ref{thm:maingeom}, reflecting the
fact that this quadratic differential is complete and has no zeros:
Each vertical trajectory (i.e.~each vertical line in $\C$) maps to a
geodesic in $\H^3$ parameterized by arc length, while each horizontal
trajectory is collapsed to a point.
\end{example}

\begin{figure}
\includegraphics[width=\textwidth]{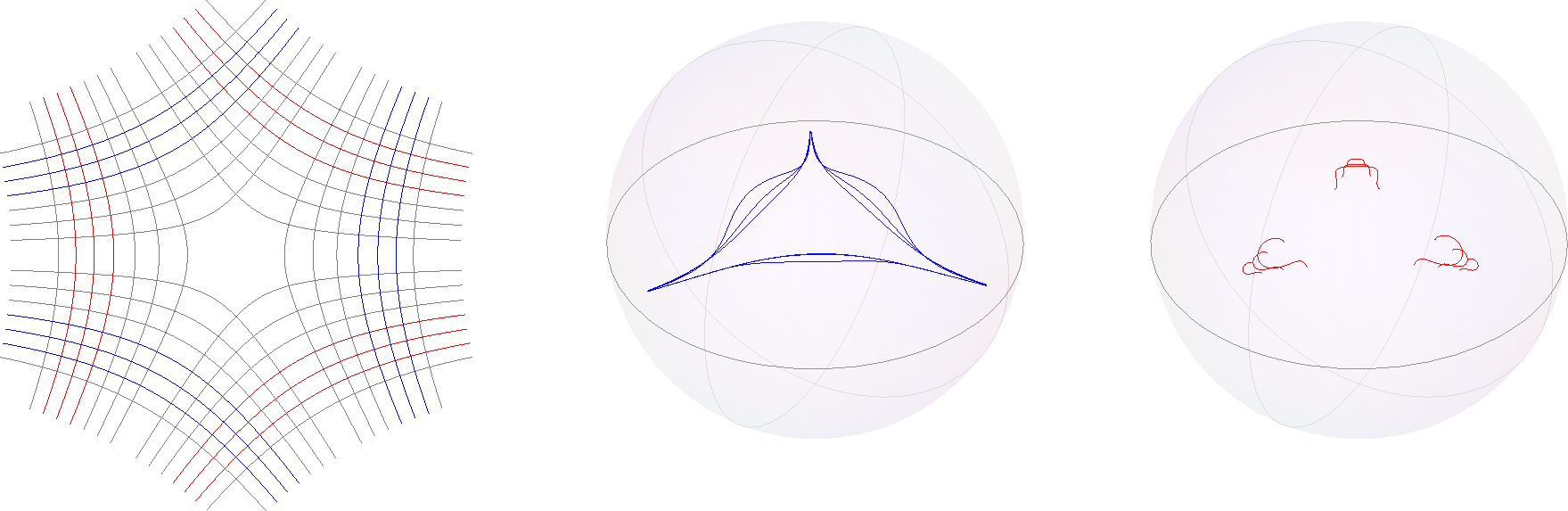}
\caption{Left to right: Vertical and horizontal trajectories of $z
  dz^2$; the images of vertical trajectories under the
  Epstein-Schwarz map approximate an ideal triangle, as shown here in the
  unit ball model of $\H^3$; segments on horizontal trajectories are
  contracted to sets of small diameter.\label{fig:epstein1}}
\end{figure}

\begin{figure}
\includegraphics[width=0.4\textwidth]{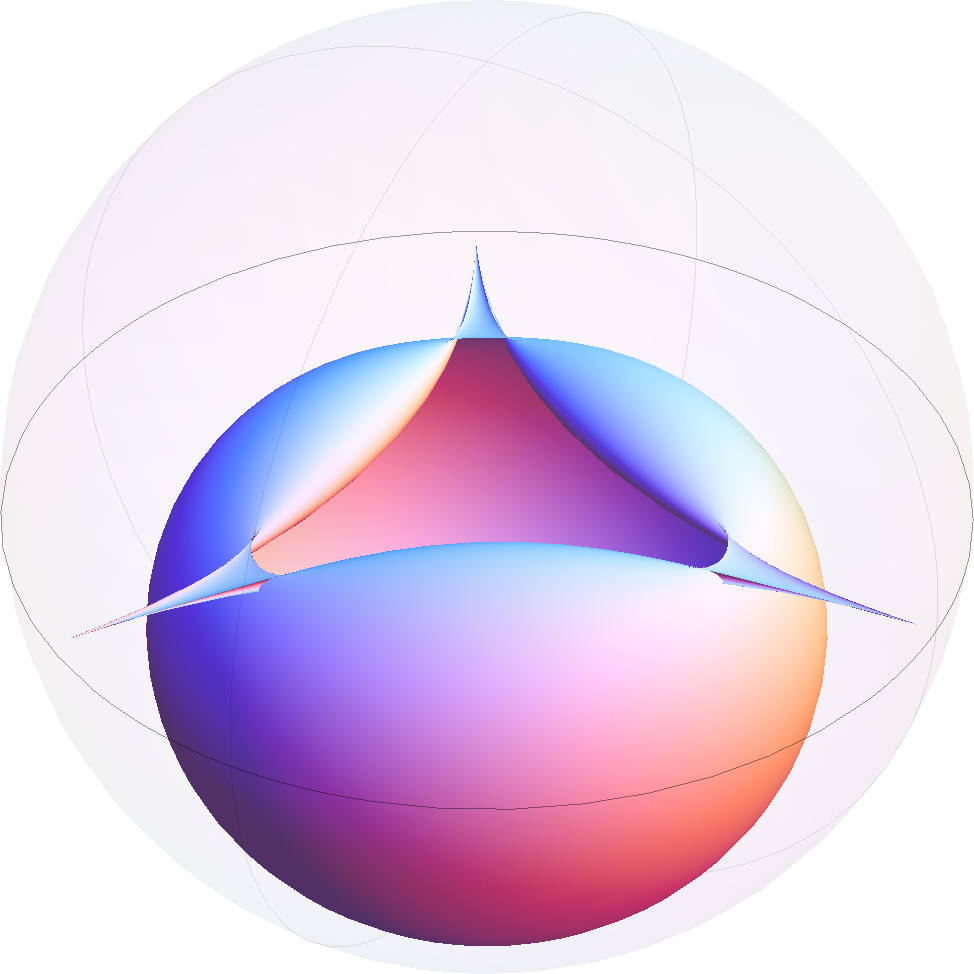}
\caption{The image of a small neighborhood of the origin under the
  Epstein-Schwarz map of $z dz^2$ in the unit ball model of $\H^3$.
  The ideal point $(0,0,-1)$ corresponds to the image of $0$ under
  the developing map.\label{fig:epstein2}}
\end{figure}

\begin{example}
Next we consider $\phi = z dz^2$ on $\C$.  While in this case it is
possible to find closed-form expressions for the developing map (in
terms of Airy functions) and for the Epstein-Schwarz map, we will
only discuss the qualitative features seen in Figures
\ref{fig:epstein1}--\ref{fig:epstein2}.  Here the origin is a simple
zero of $\phi$ which corresponds to a cone point of angle $3 \pi$ for
the $\phi$-metric.  Centered at the origin we can construct a regular
right-angled geodesic hexagon of alternating vertical and horizontal
sides.  By Theorem \ref{thm:height-estimate}, if this hexagon is far
enough from the origin then the Epstein-Schwarz map sends its
vertical sides to long near-geodesic segments in $\H^3$,
while the horizontal sides are mapped to sets of small diameter.
Thus the image of the hexagon itself approximates an ideal triangle.
Note that avoiding a small neighborhood of the origin also ensures
that the trajectories of $\phi = z dz^2$ are close to those of
$\hat{\phi} = (z + \frac{5}{8z^2}) dz^2$, so the images of these
curves approximate lines of curvature on the Epstein surface.

Near the origin (i.e.~for small $d$) the behavior of the
Epstein-Schwarz surface is quite different.  A small punctured
neighborhood of $0$ maps to the ``bubble'' shown in Figure
\ref{fig:epstein2}---a properly embedded, infinite area surface whose
induced metric is approximately isometric to
$|\hat{\phi}|/|\phi|^{1/2} \sim |z|^{-5/2} |dz|$ (by Lemma
\ref{lem:firstsecond}).  The corresponding surface in $\H^3$
approaches the developed image of $0$ tangentially, eventually leaving
every horoball based at that point.
\end{example}

\section{Sequences of Epstein-Schwarz maps}
\label{sec:sequences}

In the previous section we considered the geometry of the
Epstein-Schwarz map for a single complex projective structure on a
surface.  We now analyze how these results apply to a divergent
sequence of projective structures whose associated quadratic
differentials converge projectively.  Specifically, throughout this
section we assume:
\begin{equation}
\label{eqn:assumptions}
\begin{cases}
(f_n,\rho_n) \text{ is a sequence of projective structures on } X\\
\phi_n \in Q(X) \text{ is the associated sequence of Schwarzian
derivatives}\\
\phi_n \to \infty \text{ as } n \to \infty\\
\lim_{n \to \infty} \frac{\phi_n}{\|\phi_n\|} = \phi
\end{cases}
\end{equation}
The theme we develop is that the foliation and transverse measure of
the projective limit $\phi$ governs the large-scale geometry of
$\ep_{\phi_n}$ for large $n$.

\subsection{Nonsingular segments}

Let $\Tilde{\phi}$ and $\Tilde{\phi}_n$ denote the lifts of $\phi$ and
$\phi_n$ to $\Tilde{X}$, and let $\Tilde{Z}_\phi$ denote the set of
zeros of $\Tilde{\phi}$.  By compactness of $X$ and the convergence of
$\phi_n/\|\phi_n\|$ we have
\begin{equation}
\label{eqn:uniform-convergence}
 \frac{|\Tilde{\phi_n} - \Tilde{\phi} |}{|\Tilde{\phi}|} \to 0
\end{equation}
uniformly on compact subsets of $\Tilde{X} - \Tilde{Z}_\phi$.  

\begin{thm}
\label{thm:phin-distance}
Let $I \subset \Tilde{X}$ denote a nonsingular and non-horizontal
$\phi$-geodesic segment.  
Then there exists $N > 0$ and sequences $K_n \to 1$ and $C_n \to 0$ as
$n \to \infty$ such that for each $n > N$ and any $x_0,x_1 \in I$ with
$\phi$-height difference $h$,
we have
\begin{equation}
\label{eqn:phin-distance-goal}
 K_n^{-1} \|2\phi_n\|^{1/2} h - C_n \leq
 d_{\H^3}(\ep_{\phi_n}(x_0),\ep_{\phi_n}(x_1)) \leq 
K_n \|2\phi_n\|^{1/2} h + C_n.
\end{equation}
The constants $N,K_n,C_n$ can be taken to depend only on $\phi_n$ and
$d_{\phi}(I,Z_\phi)$.  Furthermore, the same estimate holds
for any non-horizontal half- or bi-infinite geodesic $I$ with the
property that $d_{\phi}(J,Z_\phi) > 0$.
\end{thm}

Before starting the proof, we remark that since the quantity
$\|2\phi_n\|^{1/2} h$ appearing in this estimate exactly the
$(2 \|\phi_n\| \phi)$-height of $I$, an equivalent statement is that the
$(2 \|\phi_n\| \phi)$-height parameterization of $I$ maps to a
quasigeodesic in $\H^3$ with constants $(K,C)$ converging to $(1,0)$
as $n \to \infty$.

\begin{proof}
Let $L^{(1)}$ denote the length of a subsegment of $I$ that has height
$1$.  Let $J \subset I$ denote a subsegment of length $L < L^{(1)}$,
which therefore has height $L / L^{(1)}$.

Let $U$ denote the $d/2$-neighborhood of $J$ in the $\phi$-metric,
where $d = d_\phi(J,Z_\phi)$.  Let $\phi_n^0 = \phi_n / \|\phi_n\|$.
By uniform convergence of the differentials $\phi_n^0$, for each $k
\in \N$ there exists $N(k) \in \N$ such that for $n > N(k)$ we can
apply Lemma \ref{lem:phipsi} to $J$ and $U$ with $\delta = d / (16 k
L^{(1)})$.  Thus for such $n$ there is a $\phi_n^0$-geodesic segment
with endpoints $\{x_0,x_1\}$, and we have
\begin{equation}
\begin{split}
\label{eqn:compare-length-height}
 \max(|L_n^0 - L|, |h_n^0 - h|) &< \frac{dL}{4kL^{(1)}} = \frac{dh}{4k}
\\
 d_{\phi_n^0}(J',Z_{\phi_n}) \geq d_{\phi_n^0}(J',\partial U) &>
 \frac{d}{8}
\end{split}
\end{equation}
where $L_n^0$ and $h_n^0$ are the $\phi_n^0$-length and height of the
$\phi_n^0$-geodesic segment $J'$ with endpoints $\{x_0,x_1\}$. 
Letting $L_n = \|\phi_n\|^{1/2} L_n^0$ and $h_n = \|\phi_n\|^{1/2}
h_n^0$ denote the corresponding quantities for $J'$ with respect to
$\phi_n$, and writing $d_n \define d_{\phi_n}(J',Z_{\phi_n}) >
\|\phi_n\|^{1/2}d/k$, we have
$$ \frac{d_n}{\sqrt{L_n}} > \frac{\|\phi_n\|^{1/4} d}{8 L_n^0}
> \frac{\|\phi_n\|^{1/4} d}{8(L - d/(4k))} $$ again for all $k$ and $n
> N(k)$.  By taking $n$ and $k$ large enough it follows that $d_n$ and
$d_n / \sqrt{L_n}$ can be made arbitrarily large.

For any $m > 0$, let $N'(m)$ be such that $d_n / (1 + \sqrt{L_n}) > m$
for all $n > N'(m)$.  Then for $m > M$ and $n > N'(m)$ we apply
Theorem \ref{thm:height-estimate} to $J'$,
concluding that
\begin{equation*}
 K'(m)^{-1} \sqrt{2} \, h_n - C'(m) < d_{\H^3}(\ep_{\phi_n}(x_0),\ep_{\phi_n}(x_1))
<
K'(m) \sqrt{2} \, h_n + C'(m).
\end{equation*}
Thus the parameterization of $J$ by $2\phi_n$-height is mapped by
$\ep_{\phi_n}$ to a $(K'(m),C'(m))$-quasigeodesic.  Note that the
$2\phi_n$-height of $J$ tends to $\infty$ as $n \to \infty$.

If the length of $I$ is greater than $L^{(1)}$ (e.g.~if it is a ray or
  infinite geodesic), then for $n$ sufficiently large we can apply
  Lemma \ref{lem:local-to-global} to $J$ and conclude that in any
  case, the $2\phi_n$-height parameterization of $I$ maps to a
  $(K''(m),C''(m))$-quasigeodesic, where $(K''(m),C''(m)) \to (1,0)$ as
  $m \to \infty$.

Finally we must consider the effect of changing from the
$2\phi_n$-height parameterization to the $2\phi$-height.  For the
remainder of the proof let $x_0,x_1\in I$ be an arbitrary pair
of points and let $L$ denote the $\phi$-length of the segment $[x_0,x_1]
\subset I$.  Note that it is no longer assumed that $L < L^{(1)}$.

By applying \eqref{eqn:compare-length-height}
  $\lceil L/L^{(1)} \rceil$ times we conclude that the $\phi_n$-height
  differences $h_n$ and $\phi$-height difference $h$ between $x_0$ and
  $x_1$ satisfy
$$ |\sqrt{2}\,h_n - \|2\phi\|^{1/2}h| < \frac{\|2\phi_n\|^{1/2} d L}{4kL^{(1)}} =
\left (\frac{d}{4k} \right)\|2\phi_n\|^{1/2} h,$$
where in the last step we have used that $L/L^{(1)} = h$.
Therefore, changing from the $2\phi_n$-height to the
$(\|2\phi_n\|\phi)$-height parameterization of $I$ introduces a
multiplicative error in the distance estimate that is
$o(\|\phi_n\|^{1/2}h)$ as $k \to
\infty$, and \eqref{eqn:phin-distance-goal} follows with
\begin{equation*}
\begin{split}
N &= N'(M),\\ 
K_n &= K''(m_0(n)) + d/4k_0(n),\\
C_n &= C''(m_0(n)),\\
\end{split}
\end{equation*}
where
\begin{equation*}
\begin{split}
k_0(n) &= \max \{ k \: | \: N(k) < n \},\\
m_0(n) &= \max \{ m \: | N'(m) < n \}.\\
\end{split}
\end{equation*}
\end{proof}

Complementing Theorem \ref{thm:phin-distance} we have the following
estimate for horizontal segments:

\begin{thm}
\label{thm:phin-horizontal-distance}
Let $I \subset \Tilde{X}$ denote a nonsingular $\phi$-horizontal
segment.  Then we have
$$ \diam(\ep_{\phi_n}(I)) = o(\|\phi_n\|^{1/2}) \text{ as } n \to
\infty.$$
\end{thm}

As with Theorem \ref{thm:phin-distance}, the above estimate for the
geometry of the image is not uniform---it depends on the particular
segment $I$.

\begin{proof}
We proceed in much the same way as the previous proof, but using the
  whole segment instead of a subsegment of height $1$.

  Denote by $L_n$ the $\phi_n$-length of the $\phi_n$-geodesic $I'$
  with the same endpoints as $I$ and by $d_n$ the $\phi_n$-distance
  from $I'$ to $Z_{\phi_n}$.  For large $n$ we can apply Lemma
  \ref{lem:phipsi} to a $d/2$-neighborhood $I$ in the $\phi$-metric
  with $\delta = d / (16kL)$, where $L$ is the $\phi$-length of $I$.
  Then as in the previous proof we have $m_n \to \infty$ as $n,k \to
  \infty$, where
$$ m_n \define \frac{d_n}{1 + \sqrt{L_n}}.$$
We also have $h_n^0 < \frac{d}{4k}$, or equivalently, $\sqrt{2} \, h_n < d
\|2\phi_n\|^{1/2} / (4k)$.

Let $x_0,x_1$ be the endpoints of $I$, which are also the endpoints of
$I'$. Applying Corollary
  \ref{cor:horizontal-height-estimate} if $I'$ is $\phi_n$-horizontal,
  and Theorem \ref{thm:height-estimate} if it is not, we conclude
\begin{equation*}
\begin{split}
d_{\H^3}(\ep_{\phi_n}(x_0),\ep_{\phi_n}(x_1)) &\leq K'(m_n) \sqrt{2}\,h_n +
C'(m_n)\\ &\leq \frac{K'(m) d}{4k} \|2\phi_n\|^{1/2} + C'(m_n),
\end{split}
\end{equation*}
for $n > N(k)$, and using $k = k_0(n)$ as in the proof of Theorem
\ref{thm:phin-distance} we find that the right hand side is
$o(\|\phi_n\|^{1/2})$ as $n \to \infty$.

Finally, we note that a subsegment of $I$ is shorter and its distance
from $Z_\phi$ is no less than that of $I$.  Decreasing $L$ and
increasing $d$ preserve (or improve) all of the estimates above, so
the distance estimate above applies to all pairs $x_0,x_1
\in I$.  This gives the desired bound on the diameter of
$\ep_{\phi_n}(I)$.
\end{proof}

\subsection{Periodic geodesics}

Given an element $g \in \SL_2\C$ let
$$ \ell(g) = \inf_{x \in \H^3} d(x,g \cdot x)$$
denote its translation length when acting as an isometry of $\H^3$.
Thus $\ell(g) > 0$ if and only if $g$ is a hyperbolic element, in
which case $g$ translates along its geodesic axis by distance
$\ell(g)$.

\begin{thm}
\label{thm:periodic-length}
If $\gamma \in \Pi$ is represented by
a periodic $\phi$-geodesic of height $h$, then
\begin{equation}
\label{eqn:llim}
\lim_{n \to \infty} \frac{\ell(\rho_n(\gamma))}{\|2\phi_n\|^{1/2} } = h.
\end{equation}
In particular if $h > 0$ then $\rho_n(\gamma)$ is hyperbolic for all
sufficiently large $n$.
\end{thm}

\begin{proof}
Recall that $\ep_{\phi_n}(\gamma \cdot x) = \rho_n(\gamma) \cdot
\ep_{\phi_n}(x)$.

First suppose $h>0$.  Applying Theorem \ref{thm:phin-distance} to a
nonsingular $\Tilde{\phi}$-geodesic axis of $\gamma$ in $\Tilde{X}$,
we find that $\rho_n(\gamma)$ preserves a $(K_n,C_n)$-quasigeodesic axis in
$\H^3$ along which it moves points distance $d_n$, where
$$K_n^{-1} \|2\phi_n\|^{1/2}h -
C_n < d_n <  K_n^{-1} \|2\phi_n\|^{1/2}h + C_n.$$
By Lemma \ref{lem:stability}, this quasigeodesic axis lies in a
uniformly bounded neighborhood of the geodesic axis of
$\rho_n(\gamma)$.  The translation length of $\ell(\rho_n(\gamma))$
is therefore $d_n + O(1)$ as $n \to \infty$, and since $(K_n,C_n) \to
(1,0)$, the desired estimate follows.

If $h = 0$ then we apply Theorem \ref{thm:phin-horizontal-distance} to
the horizontal $\Tilde{\phi}$-geodesic fixed by $\gamma$ to get a
pairs of points in $\H^3$ related by $\rho_n(\gamma)$ and separated by
distance $o(\|\phi_n\|^{1/2})$.  This distance is an upper bound for
$\ell(\rho_n(\gamma))$ hence $\lim_{n \to \infty}
\frac{\ell(\rho_n(\gamma))}{\|2\phi_n\|^{1/2} } = 0$ as required.
\end{proof}

\section{The character variety and growth rates}
\label{sec:properness}

In this section we show how the results and techniques of Sections
\ref{sec:epstein}--\ref{sec:sequences} can be used to study the growth
rate of the holonomy representation as a function of the Schwarzian.

\subsection{The character variety and holonomy map}
\label{sec:character-varieties}

The $\SL_2(\C)$-representation variety of $\Pi$ is the set $\sR(\Pi) =
\Hom(\Gamma,\SL_2(\C))$.  Choosing a finite generating set $\Sigma$
for $\Pi$ realizes $\sR(\Pi)$ as a closed algebraic subset of
$(\SL_2(\C))^{|\Sigma|}$, giving it the structure of an algebraic
variety.

The $\SL_2(\C)$-character variety of $\Pi$, denoted $\X(\Pi)$, is an
affine algebraic variety consisting of the characters (traces) of
representations in $\sR(\Gamma)$; there is a natural algebraic map
$\sR(\Gamma) \to \X(\Gamma)$ taking a representation to its character.
We denote this map by $\rho \mapsto [\rho]$.  The character variety
can also be described as an algebraic quotient
$$ \X(\Gamma) = \sR(\Gamma) \sslash \SL_2(\C)$$
where $\SL_2(\C)$ acts by conjugating representations.  See
\cite{culler-shalen} and
\cite[Sec.~II.4]{morgan-shalen:valuations-trees} for details about
these constructions.

As mentioned in the introduction, there is holomorphic map 
$$\hol : Q(X) \to \X(\Pi),$$
the \emph{holonomy map}, which associates to a projective structure on $X$
the character of its holonomy representation (which is well-defined,
since the representation itself is well-defined up to conjugation).

\subsection{Properness}
Gallo, Kapovich, and Marden showed that the holonomy map $Q(X) \to \X(\Pi)$ is
a proper map \cite[Thm.~11.4.1]{gkm}, following an outline presented
in \cite[Sec.~7.2]{kapovich:monodromy}.  A geometric approach to properness using pleated
surfaces can be found in \cite{tanigawa:divergence}.  The same result
also follows easily from Theorem \ref{thm:periodic-length}:

\begin{thm}
The map $\hol: Q(X) \to \X(\Pi)$ is proper.
\end{thm}

\begin{proof}
Let $\phi_n \in Q(X)$ be a divergent sequence.  By passing to a
subsequence we can assume that $\phi_n$ converges projectively,
i.e.~$\phi_n / \|\phi_n\| \to \phi$.  Let $\gamma \in \Pi$ be freely
homotopic to a periodic $\phi$-geodesic.  The translation
length of the image of $\gamma$ under a representation $\rho : \Pi \to
\SL_2(\C)$ defines a continuous function $\ell_\gamma : \X(\Pi) \to
\R$.  By Theorem \ref{thm:periodic-length} we have
$\ell_\gamma(\hol(\phi_n)) \to \infty$, so the image of the sequence
$\{ \hol(\phi_n) \}$ is not contained in a compact set.
\end{proof}

\subsection{Growth estimate}
This approach to proving properness of the holonomy map also lends
itself to effective estimates of the growth rate of holonomy
representations.  In fact, Theorem \ref{thm:periodic-length} can be
seen as an estimate of this kind, where translation length of the action
on $\H^3$ is used to measure the ``size'' of a representation. Since
translation length grows logarithmically with respect to trace
coordinates on $X(\Pi)$, the holonomy map itself has exponential
growth in these coordinates.  Making this coordinate-independent, we
have the following:

\begin{thm}[Effective properness]
\label{thm:effective-properness}
For any affine embedding $\X(\Pi) \into \C^n$ and any norm $\| \param
\|$ on $\C^n$ there are constants $A > 0$ and $B$ such that
\begin{equation}
\label{eqn:properness}
 A^{-1} \|\phi\|^{1/2} - B < \log( 1 + \|\hol(\phi)\| ) < A
\|\phi\|^{1/2} + B.
\end{equation}
\end{thm}

In \cite{simpson}, Simpson uses harmonic maps techniques to obtain a
similar bound for the growth rate of the map from the \emph{de Rham moduli
space} of rank-$2$ systems of ordinary differential equations over a
compact Riemann surface to the character variety of the fundamental
group.  It would be interesting to know whether the set of projective
structures is properly embedded in this moduli space of ODEs, and thus
to see if the growth rate of holonomy in terms of the norm of the
Schwarzian can also be estimated by Simpson's technique.

The proof of Theorem \ref{thm:effective-properness} will depend on an
estimate that is a direct analog of Theorem \ref{thm:periodic-length},
but where we consider a fixed homotopy class of curves and an
arbitrary quadratic differential, instead of a fixed sequence of
quadratic differentials and an arbitrary homotopy class.

\begin{thm}
\label{thm:length-bounds}
For each $\gamma \in \Pi$ there exists constants $C>0$ and $N>0$ such
that if $\phi \in Q(X)$ satisfies $\|\phi\| > N$ then
$$ \ell(\rho_\phi(\gamma)) \leq C \|\phi\|^{1/2}. $$
Furthermore, if $\gamma$ is represented by a periodic $\phi$-geodesic
of height $h$, angle $\theta > 0$, and whose associated flat annulus has
width at least $w \|\phi\|^{1/2}$, then
$$ \ell(\rho_\phi(\gamma)) \geq c \|\phi\|^{1/2},$$
where in this case $c$ and $N$ also depend on $\gamma$, $\theta$, and $w$.
\end{thm}

\begin{proof}
The unit sphere in $Q(X)$ corresponds to a compact family of metrics
on $X$.  Thus the free homotopy class of an element $\gamma \in \Pi$
can be realized by a curve of uniformly bounded length and height with respect to
any $\phi$ such that $\|\phi\|=1$.  Increasing length and height by a
bounded amount we can further assume that each such realization avoids
a fixed neighborhood of $Z_\phi$.

Scaling to obtain $\phi \in Q(X)$ of any norm, we conclude that
$\gamma$ is represented by a closed curve in $X$ of length bounded by
$C' \|\phi\|^{1/2}$ and which avoids a $\delta
\|\phi\|^{1/2}$-neighborhood of $Z_\phi$, for some constants
$C',\delta$.  We can lift this closed curve to a path in $\Tilde{X}$
whose endpoints are identified by the action of $\gamma$.  Choosing
$N$ large enough we can apply Lemma 3.6 to conclude that $\ep_\phi$ is
uniformly Lipschitz on this path, so the image in $\H^3$ has length
bounded by $C''\|\phi\|^{1/2}$.  Since the endpoints of the image are
identified by $\rho_\phi(\gamma)$, this gives the desired upper bound
for $\ell(\rho_\phi(\gamma))$.

For the periodic case we can again use compactness of the unit sphere
in $Q(X)$ and the angle $\theta$ to obtain a lower bound on the
$\phi$-height of a periodic geodesic homotopic to $\gamma$ of the form
$h > c' \|\phi\|^{1/2}$, where $c'$ depends on $\gamma$ and $\theta$.
Of course the length estimate $L < C' \|\phi\|^{1/2}$ applies as above.
Using the geodesic representative in the center of the flat annulus,
the distance from this geodesic to the nearest zero of $\phi$ is at
least $d = \frac{1}{2} w \|\phi\|^{1/2}$.

For $\|\phi\| > N$ and $N$ sufficiently large (now depending on
$\gamma$, $\theta$, and $w$), we have $d > M (1 + \sqrt{L})$ where $M$
is the constant from Theorem \ref{thm:height-estimate}.  Then
\eqref{eqn:distance-estimate} shows that the lift of the periodic
geodesic to $\Tilde{X}$ maps by $\ep_\phi$ to a uniformly
quasigeodesic axis for $\rho_\phi(\gamma)$ in $\H^3$ on which the
translation length is bounded below by a multiple of the height $h$.
Using the stability of quasigeodesics in $\H^3$ (Lemma
\ref{lem:stability}) we obtain a lower bound of the form
$\ell(\rho_\phi(\gamma)) > c'' \|\phi\| - D$.  The lower bound on
$\|\phi\|$ allows us to remove the additive constant by changing the
multiplicative factor slightly, and the Theorem follows.
\end{proof}

\begin{proof}[Proof of Theorem \ref{thm:effective-properness}.]
Let $P \subset \Pi$ and $w_0$ be as in Theorem
\ref{thm:finite-periodic}.  Since traces of elements of $\Pi$ are
regular functions on $\X(\Pi)$, the traces of elements of $P$ have a
uniformly polynomial upper bound in the coordinates of the affine
embedding.  Thus there are constants $C,k$ such that for all $\gamma
\in P$ we have
$$ |\tr(\rho_\phi(\gamma))|  \leq C ( 1 + \| \hol(\phi) \|)^k.$$

For each $\phi \in Q(X)$ there exists $\gamma \in P$ that is
represented by a periodic $\phi$-geodesic that is nearly vertical and
thus has height bounded below by $c \|\phi\|^{1/2}$ for some positive
constant $c$.  Since we also have a uniform lower bound on the widths
of the corresponding flat annuli, Theorem \ref{thm:length-bounds}
and the relation between trace and translation length give
$$ |\tr(\rho_\phi(\gamma))| > \exp(c' \|\phi\|^{1/2})$$
for some $c' > 0$, as long as $\|\phi\| > M$.  Here we have uniform
constants because $P$ is finite.  Combining this with the
previous inequality and taking logarithms gives the lower bound on
$\|\hol(\phi)\|$ from \eqref{eqn:properness}, where adjusting the
additive constant $B$ allows us to remove the requirement that
$\|\phi\|$ is large.

The upper bound from \eqref{eqn:properness} is similar, but easier:
The ring of regular functions on $\X(\Pi)$ is generated by the trace
functions of finitely many elements of $\Pi$ (see
\cite[Sec.~1.4]{culler-shalen}), so $\|\hol(\phi)\|$ has a polynomial
upper bound in terms of these traces.  Applying the upper bound on
translation length from Theorem \ref{thm:length-bounds} to these
elements and again taking logarithms completes the proof.
\end{proof}

\section{The Morgan-Shalen compactification and straight maps}
\label{sec:final}

\subsection{The compactification}

Consider the map $\X(\Pi) \to (\R^+)^{\Pi}$ given by
$$ [\rho] \mapsto \left ( \log (|\tr \rho(\gamma)| + 2) \right
)_{\gamma \in \Pi}. $$
Let $\P(\R^+)^{\Pi}$ denote the space of rays in $(\R^+)^{\Pi}$ and
consider the projectivized map $\X(\Pi) \to \P(\R^+)^{\Pi}$.  The
image of $\X(\Pi)$ is precompact and the closure of the image defines
the \emph{Morgan-Shalen compactification} of $\X(\Pi)$.  If $\ell :
\Pi \to \R^+$ is a function whose projective class $[\ell]$ is a
boundary point of $\X(\Pi)$, then there exists an $\R$-tree $T$ and
an isometric action of $\Pi$ on $T$ such that
$$ \ell(\gamma) = \inf_{x \in T} d(x, \gamma \cdot x), $$
that is, $\ell$ is the translation length function of an action of
$\Pi$ on an $\R$-tree.  As in the introduction we say in this case
that $T$ \emph{represents} $[\ell]$.

This compactification was introduced in
\cite{morgan-shalen:valuations-trees} where a tree representing a
boundary point is described in terms of a valuation on the function
field of $\X(\Pi)$.  For our purposes it will be important to
construct such a tree directly from the action of a representation on
hyperbolic space, so we will use an alternative construction of
representing trees based on the asymptotic cone.

\subsection{Asymptotic cone construction}

Bestvina \cite{bestvina} and Paulin \cite{paulin} used geometric limit
constructions to build $\R$-trees representing limit points of
sequences of representations in the Morgan-Shalen compactification.
Later, Chiswell \cite{chiswell} and Kapovich-Leeb
\cite{kapovich-leeb} described how these limit constructions can be
interpreted in terms of asymptotic cones of hyperbolic spaces.  We now
review this approach, mostly following the exposition of
\cite[Ch.~9--10]{kapovich:book}.

Fix a non-principal ultrafilter $\omega$ on $\mathbb{N}$ and denote by
$\olim a_n$ the $\omega$-limit of a sequence of real numbers
$\{a_n\}$.    Given a metric space $X$, a sequence of points $c_n \in X$, and
a sequence $\epsilon_n \to 0$ of positive reals, we denote by
$$ \olim \: (\epsilon_n X, c_n)$$
the \emph{asymptotic cone} of $X$ based at $x_n$ with scale factors
$\epsilon_n$; this is the quotient metric space associated with the
set of sequences 
$$\{ x = (x_n)\: | \: x_n \in X, \; \olim \epsilon_n d(c_n,x_n) <
\infty \}$$
and the pseudometric
$$ d(x,y) = \olim \epsilon_n d(x_n,y_n).$$  If $X$ is a $\CAT(\kappa)$ space for some $\kappa<0$ (for example,
$X=\H^3$) then $\olim \: (\epsilon_n X, c_n)$ is an $\R$-tree.

Now fix a finite generating set $\Sigma$ for $\Pi$.  If $\rho :
\Pi \to \isom(X)$ is an isometric action, we define the \emph{local
  scale} of $\rho$ at $x$ to be the quantity
$$R(\rho,x) = \max_{\gamma \in \Sigma} d(x, \rho(\gamma) \cdot x).$$

Specializing to the case of $X = \H^3$, the basic link between the
asymptotic cone construction and the Morgan-Shalen compactification is
the following (see \cite[Sec.~10.4]{kapovich:book}):

\begin{thm}
\label{thm:morgan-shalen-asymptotic-cone}
Consider a sequence $\rho_n \in \X(\Pi)$ and identify it with a
sequence of isometric actions of $\Pi$ on $\H^3$ using the covering
$\SL_2\C \to \PSL_2\C \simeq \isom^+(\H^3)$.  Let $c_n \in \H^3$ be a
sequence of points and $\epsilon_n \to 0$ a sequence of positive
reals.
\begin{rmenumerate}
\item 
If $\olim \epsilon_n R(\rho_n,x_n) < \infty$, then the action
$$\gamma : (x_n) \mapsto (\rho_n(\gamma) \cdot x_n)$$
of $\Pi$ on sequences in $\H^3$ induces an isometric action of $\Pi$ on
the $\R$-tree $T := \olim \: (\epsilon_n \H^3,c_n)$.  

\item If the action of $\Pi$ on $T$ does not have a global fixed
point, and if the sequence $[\rho_n]$ converges in the Morgan-Shalen
compactification, then $T$ represents the Morgan-Shalen limit.
\end{rmenumerate}
\noproof
\end{thm}

The result above is proved in \cite[Sec.~10.4]{kapovich:book}, though
the statements of the theorems in that section are structured somewhat
differently from the one above.  Kapovich makes a specific choice of
basepoints $c_n$, but this choice is only used to show the resulting
action has no global fixed point, which we do not claim here.  (That
an arbitrary sequence of basepoints can be used is also established in
\cite[Prop.~4.8]{bestvina}.)  Similarly, while Kapovich fixes
$\epsilon_n = R(\rho_n,c_n)^{-1}$, the arguments use only that $\olim
\epsilon_n R(\rho_n,x_n) < \infty$.

A key feature of this construction of a limit tree is that it gives a
notion of convergence of a sequence $x_n \in \H^3$ to a point in $T$,
which is simply a restatement of the definition: A point
$x_\infty \in T$ is an equivalence class of sequences in $\H^3$, and
we say $x_n \to x_\infty$ if the sequence $(x_n)$ lies in that
equivalence class.  This allows us to consider the question of whether
a sequence of maps into $\H^3$ converges pointwise to a map into $T$.

\subsection{Convergence of Epstein-Schwarz maps}

We return to the hypotheses of \eqref{eqn:assumptions}, that is,
considering a divergent sequence of projective structures with
holonomy representations $\rho_n$ and quadratic differentials $\phi_n$
converging projectively to $\phi$.  We suppose also that $[\rho_n]$
converges in the Morgan-Shalen compactification to the projective
equivalence class $[\ell]$ of a function $\ell : \Pi \to \R^+$.

Recall also that $\Tilde{Z}_\phi \subset \Tilde{X}$ is the discrete
subset of the universal cover of $X$ consisting of points that project
to zeros of $\phi$, and similarly for $\Tilde{Z}_{\phi_n}$ and
$\phi_n$.

\begin{thm}
\label{thm:limit-of-epstein-maps}
Fix a point $z_0 \in \Tilde{X} \setminus \Tilde{Z}_\phi$ and use
its $\ep_{\phi_n}$-images as basepoints to construct the
asymptotic cone
$$ T = \olim \: (\|2\phi_n\|^{-1/2} \H^3, \ep_{\phi_n}(z_0)).$$
Then:
\begin{rmenumerate}
\item The sequence of maps $\ep_{\phi_n} : \Tilde{X} \setminus
\Tilde{Z}_{\phi_n} \to \H^3$ converges (pointwise) to a continuous map
$\ep_\infty : \Tilde{X} \to T$.
\item For any pair of points $x,y \in \Tilde{X}$ that are endpoints of
a nonsingular $\Tilde{\phi}$-geodesic segment of height $h$, the map
$\ep_\infty$ satisfies
$$d(\ep_\infty(x), \ep_\infty(y)) = h.$$
\item The sequence $\rho_n$ induces an isometric action of $\Pi$ on
$T$ which represents the Morgan-Shalen limit $[\ell]$, and
$\ep_\infty$ is equivariant for this action.
\end{rmenumerate}
\end{thm}

Note that in statement (i) the domains of the maps $\ep_{\phi_n}$ vary
with $n$ so the limit is \emph{a priori} only defined on
$$ \liminf_{n \to \infty} \: (\Tilde{X} \setminus \Tilde{Z}_{\phi_n}),$$
which contains $\Tilde{X} \setminus \Tilde{Z}_\phi$ since
$\phi_n/\|\phi_n\| \to \phi$.  However we will show that the limit map on
$\Tilde{X} \setminus \Tilde{Z}_\phi$ has a unique continuous extension
to $\Tilde{X}$.

\begin{proof}
Suppose $z,z' \in \Tilde{X} \setminus \Tilde{Z}_\phi$.  Join these
points by a polygonal path in $\Tilde{X} \setminus \Tilde{Z}_\phi$
that is a finite union of nonsingular $\phi$-geodesic segments.
Applying Theorems
\ref{thm:phin-distance} and \ref{thm:phin-horizontal-distance} to the
segments and using the triangle inequality we find
\begin{equation}
\label{eqn:distance-and-height}
 \limsup_{n \to \infty} \|2\phi_n\|^{-1/2} d(\ep_{\phi_n}(z),\ep_{\phi_n}(z')) \leq
\bar{h}
\end{equation}
where $\bar{h}$ is the sum of the $\phi$-heights of the segments.
Furthermore, if there is only one segment then the limit exists and is
equal to $\bar{h}$.

Applying this to $z' = z_0$ it follows that $\olim \|2\phi_n\|^{-1/2}
d(\ep_{\phi_n}(z),\ep_{\phi_n}(z_0))$ is finite, thus the sequence $\ep_\infty(z) :=
(\ep_{\phi_n}(z))$ represents a point of the asymptotic cone $T$.
This gives the desired pointwise limit map on $\Tilde{X} \setminus
\Tilde{Z}_\phi$.  Equality of the limit
\eqref{eqn:distance-and-height} in the one-segment case is exactly
statement (ii).

The polygonal path chosen above can be taken to agree with the
minimizing $\Tilde{\phi}$-geodesic joining $z$ to $z'$ except in an
arbitrarily small neighborhood $\Tilde{Z}_\phi$ where the polygonal
path must make short detours to avoid the zeros.  As a result, we can
assume that $\bar{h}$ is as close as we like to the
$\Tilde{\phi}$-height of the minimizing geodesic, which is itself a
lower bound for the $\Tilde{\phi}$-distance from $z$ to $z'$.
Therefore, the estimate above also shows that the limit $\ep_\infty
: (\Tilde{X} \setminus \Tilde{Z}_\phi) \to T$ is $1$-Lipschitz for
that metric, and in particular continuous.  Furthermore, the
asymptotic cone $T$ is a complete metric space (see
e.g.~\cite[Lem.~5.53]{bridson-haefliger}) so the Lipschitz map
$\ep_\infty$ extends uniquely and continuously to the metric
completion of its domain, which is $\Tilde{X}$.  Statement (i) follows.

From the asymptotic cone construction it is immediate that a limit of
equivariant maps is equivariant, as long as the group action is
defined on the asymptotic cone.  Thus statement (iii) is exactly the
conclusion of Theorem \ref{thm:morgan-shalen-asymptotic-cone} once we
establish the relevant hypotheses, i.e.
\begin{enumerate}
\item $\olim \|2\phi_n\|^{-1/2} R(\rho_n,\ep_{\phi_n}(z_0)) < \infty$
\item $\Pi$ acts on $T$ without global fixed points.
\end{enumerate}
Estimate (1) follows from \eqref{eqn:distance-and-height} since each
element of the finite generating set for $\Pi$ can be represented by a
polygonal path in $\Tilde{X} \setminus \Tilde{Z}_\phi$ based at $z_0$.  The total
height of this collection of paths is then a bound for $\limsup_{n \to \infty}
\|2\phi_n\|^{-1/2} R(\rho_n,\ep_{\phi_n}(z_0))$ and thus for the $\omega$-limit as
well.  Hence $\Pi$ acts on $T$.

Now suppose for contradiction that there is a point $x \in T$ fixed by
$\Pi$.  Then $x$ is the equivalence class of a
sequence $(x_n) \subset \H^3$ such that $\olim \|2\phi_n\|^{-1/2} d(x_n,
\rho_n(\gamma) \cdot x_n) = 0$ for all $\gamma \in \Pi$.  Thus
$$ \liminf_{n \to \infty} \|2\phi_n\|^{-1/2} \ell(\rho_n(\gamma)) \leq
\liminf_{n \to \infty} \|2\phi_n\|^{1/2} d(x_n,
\rho_n(\gamma) \cdot x_n) = 0.$$
But this contradicts Theorem \ref{thm:periodic-length} for any
$\gamma \in \Pi$ which can be represented by a periodic and non-horizontal
$\phi$-geodesic, and such elements exist by Theorem \ref{thm:dense}.
This contradiction shows (2), completing the proof of statement (iii).
\end{proof}

\subsection{Dual trees of quadratic differentials}
Given a measured foliation of a surface, we can lift the foliation to
the universal cover and consider the space of leaves; the transverse
measure of the foliation induces a metric on this leaf space, making
it an $\R$-tree on which $\Pi$ acts by isometries (see
\cite{morgan-shalen:free-actions} \cite[Sec.~11.12]{kapovich:book} for
details).  Applying this construction to the horizontal foliation
$\F(\phi)$ of a quadratic differential $\phi \in Q(X)$ gives the
\emph{dual tree} $T_\phi$.  By construction we also have a projection
map $\pi : \Tilde{X} \to T_\phi$.

\subsection{Straight maps}
\label{sec:straight}

A nonsingular $|\phi|$-geodesic segment in $\Tilde{X}$ of height $h$
maps by $\pi$ to a geodesic segment of length $h$ (or a point, if
$h=0$) in $T_\phi$.  We say that a segment in $T_\phi$ is
\emph{nonsingular} if it arises in this way.  (Note that a nonsingular
segment in $T_\phi$ might also arise as the image of a geodesic in
$\Tilde{X}$ that contains singularities, because a given path in
$T_\phi$ can have many geodesic lifts through $\pi$.)

We say that a map $F: T_\phi \to T$ is \emph{straight} if its
restriction to every nonsingular segment in $T_\phi$ is an isometric
embedding.  Evidently an isometry is a straight map, though the
converse does not hold (see e.g.~Lemma \ref{lem:straight-to-R} below).
Because any segment in $T_\phi$ can be lifted to a path in $\Tilde{X}$
that is piecewise geodesic, straight maps are \emph{morphisms} of
$\R$-trees in the sense of \cite{skora}.

We will use the following criterion for recognizing straight maps:

\begin{lem}
\label{lem:straightness}
Let $T$ be an $\R$-tree and $f : \Tilde{X} \to T$ a continuous map
such that for every nonsingular $\phi$-geodesic segment $J$ in
$\Tilde{X}$ with height $h$ and endpoints $x,y$, we have
\begin{equation}
\label{eqn:distance-is-height}
d(f(x),f(y)) = h.
\end{equation}
Then the map $f$ factors as $f = F \circ \pi$ where $F : T_\phi \to T$
is straight and $\pi : \Tilde{X} \to T_\phi$ is the projection.
Furthermore, if $f$ is equivariant with respect to an action of $\Pi$
on $T$, then $F$ is also equivariant.
\end{lem}

\begin{proof}
Condition \eqref{eqn:distance-is-height} implies that $f$ it is
constant on all nonsingular horizontal leaf segments.  By continuity,
it is also constant on segments of horizontal leaf segments with
endpoints at zeros, and therefore on all horizontal leaves (including
those which pass through zeros of $\phi$).  By construction of $\pi :
\Tilde{X} \to T_\phi$ as a quotient map, this is equivalent to having
a unique factorization $f = F \circ \pi$ where $F : T_\phi \to T$ is
continuous.

The parameterization of a $\phi$-geodesic in $\Tilde{X}$ by height
maps by $\pi$ to a geodesic segment in $T_\phi$ parameterized by arc
length.  Thus \eqref{eqn:distance-is-height} shows that $F$ is an
isometric embedding when restricted to a nonsingular segment in
$T_\phi$, i.e.~the map $F$ is straight.

Equivariance of $F$ follows from that of $f$ by uniqueness of the
factorization.
\end{proof}

\subsection{Proof of Theorem \ref{thm:main-folding}}

We have a divergent sequence $\phi_n$ with projective limit $\phi$ and
an accumulation point $[\ell]$ of $\hol(\phi_n)$ in the Morgan-Shalen
boundary of $\X(\Pi)$.  Pass to a subsequence (still called $\phi_n$)
so that $\hol(\phi_n)$ converges to $[\ell]$.  Theorem
\ref{thm:limit-of-epstein-maps} gives an $\R$-tree representing
$[\ell]$, which we denote by $T_0$, and an equivariant map $\ep_\infty
: \Tilde{X} \to T_0$.  Let $T = \ep_\infty(\Tilde{X})$ denote the
image of this map, which by equivariance is also an $\R$-tree carrying
an isometric action of $\Pi$.  Passing to an invariant subtree does
not change the translation length function of a group action
(\cite[II.2.2 and II.2.12]{morgan-shalen:valuations-trees}), so $T$
also represents $[\ell]$.  Part (ii) of Theorem
\ref{thm:limit-of-epstein-maps} shows that the surjective map
$\ep_\infty : \Tilde{X} \to T$ satisfies the hypotheses of Lemma
\ref{lem:straightness} and hence gives a surjective, equivariant
straight map $T_\phi \to T$. \hfill \qedsymbol

\subsection{Simple zeros} For dual trees of quadratic differentials
with only simple zeros, straight maps are isometric:

\begin{lem}
\label{lem:simple}
If $\phi \in Q(X)$ has only simple zeros, then any straight map $F :
T_\phi \to T$ is an isometric embedding.  In particular, if $\Pi$ acts
minimally on $T$ and $F$ is equivariant, then $T$ is equivariantly
isometric to $T_\phi$.
\end{lem}

\begin{figure}
\begin{center}
\includegraphics[height=4cm]{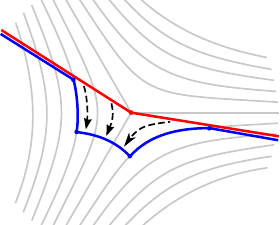}
\caption{A geodesic passing through a simple zero can be pushed to a
  nonsingular segment by an isotopy that moves along leaves of the
  horizontal foliation.\label{fig:pushoff}}
\end{center}
\end{figure}

The proof rests on a well-known technique of deforming a
$\phi$-geodesic so that it avoids a neighborhood of the zeros (compare
e.g.~\cite[Lem.~4.6]{wolf:thesis}), which for simple zeros can be
accomplished without changing the image in the dual tree.  The
specific construction we use here closely parallels that of Farb-Wolf
in \cite[Sec.~5.2]{Farb-Wolf}.

\begin{proof}
A local homeomorphism from an interval in $\R$ to an $\R$-tree is in
fact a homeomorphism and its image is a geodesic.  Consider a pair of
points $x,y \in T_\phi$ and lifts $\Tilde{x},\Tilde{y} \in \Tilde{X}$
through the projection $\pi : \Tilde{X} \to T_\phi$.  Let $J$ be the
$\phi$-geodesic joining $\Tilde{x}$ and $\Tilde{y}$, which
consists of a sequence of nonsingular segments that meet at zeros of
$\phi$.

Since $F$ is straight, its restriction to $\pi(J)$ maps each
nonsingular segment onto a geodesic in $T$, and the sum of the lengths
of these geodesics is $d(x,y)$.  If we show that $\left . F \right
|_{\pi(J)}$ is also locally injective near the image of a zero of
$\phi$, then $f(J)$ is the geodesic from $f(x)$ to $f(y)$ and we
conclude that $d(x,y) = d(f(x),f(y))$ for all $x,y \in T$.

If a $\phi$-geodesic $J \subset \Tilde{X}$ passes through a zero $z$
of $\phi$, then sum of the angles on either side of $J$ at $z$ is $(k
+ 2) \pi$, where $k$ is the order of the zero.  Thus at a
\emph{simple} zero, there is a side on which the angle is less than $2
\pi$.  On this side, we can push the part of $J$ near $z$ to a
nonsingular segment of a vertical leaf by an isotopy that moves along
horizontal leaves of $\phi$ (see Figure \ref{fig:pushoff}).  In
particular the segment of $\pi(J)$ near $\pi(z)$ is also the image of
a nonsingular segment in $\Tilde{X}$.  Since a straight map is
injective on such segments, we conclude that $\left . F \right
|_{\pi(J)}$ is locally injective, as desired.
\end{proof}

\subsection{Proof of Theorem \ref{thm:main-simple}}
Let $[\ell]$ be an accumulation point of $\hol(\phi_n)$.  Theorem
\ref{thm:main-folding} gives a tree $T$ representing $[\ell]$ and a
straight map $T_\phi \to T$.  By Lemma \ref{lem:simple} the straight
map is an isometric embedding and hence $[\ell]$ is the length
function of the action of $\Pi$ on $T_\phi$.  In particular there is
only one accumulation point of this sequence in the Morgan-Shalen
compactification.  Furthermore, by \cite[Thm.~3.7]{culler-morgan}, any
$\R$-tree on which $\Pi$ acts isometrically with this length function
has a unique minimal invariant subtree equivariantly isometric to
$T_\phi$.

The set of quadratic differentials that have a zero of multiplicity at
least $2$ is a closed algebraic subvariety of $Q(X) \simeq \C^{3g-3}$,
so this set is nowhere dense and null for the Lebesgue measure class.
This gives the required properties for the set of differentials with
only simple zeros.  \hfill \qedsymbol

\subsection{Abelian actions and straight maps}

An \emph{abelian} action of $\Pi$ on an $\R$-tree is one which has
nonzero translation length function $\ell : \Pi \to \R$ of the form
$\ell(g) = |\chi(g)|$ where $\chi : \Pi \to \R$ is a homomorphism.
(See \cite{alperin-bass} for detailed discussion of such actions.)  The
homomorphism $\chi$ can be recovered, up to sign, from the length
function $\ell$.  The action of $\Pi$ on $\R$ by translations given by
$g \cdot x = x + \chi(g)$, is an example of an abelian action, which
we call the \emph{shift} induced by $\chi$.

An abelian action on an $\R$-tree fixes an end of the tree, and the
Busemann function of this end gives an equivariant map $b: T \to
\R$ that intertwines the action of $\Pi$ on $T$ with the shift induced
by $\chi$.  Thus the shift is ``final'' among actions with a given
abelian length function.

Straightness is also preserved by composition with the Busemann
function of an abelian action:
\begin{lem}
\label{lem:straight-to-R}
  Let $T$ be an $\R$-tree equipped with an abelian action of
  $\Pi$ by isometries, and let $b : T \to \R$ denote the
  Busemann function of a fixed end.  If $F : T_\phi \to T$ is an
  equivariant straight map, then $b \circ F$ is also straight.
\end{lem}

\begin{proof}
  Let $\gamma \in \Pi$ be an element represented by a periodic
  $\phi$-geodesic.  This periodic geodesic lifts to a complete
  geodesic axis $\Tilde{L} \subset \Tilde{X}$ on which $\gamma$ acts as a
  translation, and $L \define \pi(L)  \subset T_\phi$ is the axis of the action
  of $\gamma$ on $T_\phi$.

  Because $F$ is $\phi$-straight, it maps $L$ homeomorphically to
  the geodesic axis of $\gamma$ in $T$.  Since $F(L)$ is
  $\gamma$-invariant, in one direction it is asymptotic to the fixed
  end of $\Pi$ on $T$, and the restriction of $b$ to $F(L)$ is an
  isometry.  Thus $b \circ f$ maps any segment along $L$ of height $h$
  to an interval in $\R$ of length $h$.

  Now consider an arbitrary nonsingular $\phi$-geodesic segment $J
  \subset \Tilde{X}$ with endpoints $\Tilde{x},\Tilde{y}$ and height
  $h$.  By Theorem \ref{thm:dense}, periodic $\phi$-geodesics are
  dense in the unit tangent bundle of $X$, so we can approximate $J$
  by a segment on an axis of some element $\gamma \in \Pi$ in
  $\Tilde{X}$.  More precisely, we can find such $\Tilde{L}$ and a
  pair of points $\Tilde{x}',\Tilde{y}' \in \Tilde{L}$ such that the
  pairs $(\Tilde{x},\Tilde{x}')$ and $(\Tilde{y},\Tilde{y}')$
  determine nonsingular horizontal $\phi$-geodesic segments.  Let $x =
  \pi(\Tilde{x})$ and similarly for $y$, $x'$, and $y'$.  Then $F(x) =
  F(x')$, $F(y) = F(y')$, and by the previous argument we have $|
  b(F(x)) - b(F(y))| = h$.  Thus $\pi(J)$ maps by $b \circ F$ to a
  segment of length $h$, and $b \circ F$ is straight.
\end{proof}

\begin{lem}
\label{lem:harmonic-representative}
  Let $T$ be an $\R$-tree equipped with an abelian action of
  $\Pi$ by isometries with length function $\ell = |\chi|$.  If
  there exists a $\phi$-straight map $f : \Tilde{X} \to T$, then $\phi
  = \omega^2$ where $\omega$ is the holomorphic $1$-form on $X$ whose
  imaginary part is the harmonic representative the cohomology class
  of $\chi : \Pi \to \R$.
\end{lem}

Note that this lemma is an analog for straight maps of the properties of
harmonic maps established in \cite[Thm.~3.7]{ddw:morgan-shalen}, and our
technique is a straightforward adaptation of their argument.

\begin{proof}
  By the previous lemma, we can assume that $T = \R$ with the shift
  action induced by $\chi$.  In this case it suffices to show that $f
  : \Tilde{X} \to \R$ is a harmonic function with $\Tilde{\phi} = 4
  (\partial f)^2$, for then $\Tilde{\omega} = 2 \partial f$ is
  $\Pi$-invariant and descends to a $1$-form on $X$ which, by
  construction, has periods (and thus cohomology class) given by the
  translation action of $\chi$.

  Away from the zeros of $\Tilde{\phi}$, we have a local conformal
  coordinate $z$ for $\Tilde{X}$ in which $\Tilde{\phi} = dz^2$.
  Restricting $f$ to such a coordinate neighborhood and considering it
  as a function of $z$, the $\phi$-straightness condition implies that
  $f$ if constant on horizontal lines and on a vertical line it has
  the form $\pm \Im(z) + C$ for some constant $c$.  In particular $f$
  is a real linear function and $\partial f = \pm 2 dz$.  Thus in a
  neighborhood of any point in $\Tilde{X}$ that is not a zero of
  $\Tilde{\phi}$, we can express $f$ as the composition of the
  conformal coordinate map $z$ and a real linear function, which is
  harmonic.  Since the zeros of $\Tilde{\phi}$ are isolated and $f$ is
  continuous (thus bounded in a neighborhood of each zero), the
  function $f : \Tilde{X} \to \R$ is harmonic.  The equation
  $\Tilde{\phi} = 4 (\partial f)^2$, which we have verified away from
  the zeros, also extends by boundedness of $f$.
\end{proof}

\subsection{Proof of Theorem \ref{thm:main-abelian}}
We have a divergent sequence $\phi_n$ such that $\hol(\phi_n)$
converges in the Morgan-Shalen sense to an abelian length function
$|\chi|$.  Consider any subsequence $\phi_{n_k}$ that converges
projectively, to $\phi \in Q(X)$.  As in the proof of Theorem
\ref{thm:main-folding}, there is a subsequence of $\hol(\phi_{n_k})$
giving a limit action on an $\R$-tree $T_\infty$ representing
$|\chi|$ and an equivariant straight map $F : T_\phi \to
T_\infty$.  By Lemma \ref{lem:harmonic-representative} we have $\phi =
\omega^2$ where $\omega$ is the harmonic representative of $[\chi]$.

Since we have shown that this is the unique projective accumulation
point of the original sequence, we conclude that $\phi_n$ converges
projectively to $\omega^2$.

\nocite{} 
\newcommand{\removethis}[1]{}

\end{document}